\documentclass[12pt]{amsart}
\usepackage{amssymb}
\usepackage{mathtools}
\usepackage[english, french]{babel}
\usepackage{epsfig}
\numberwithin{equation}{section}

\def\v{\varphi}
\def\Re{{\sf Re}\,}
\def\Im{{\sf Im}\,}

\def\1#1{\overline{#1}}
\def\2#1{\widetilde{#1}}
\def\3#1{\widehat{#1}}
\def\4#1{\mathbb{#1}}
\def\5#1{\frak{#1}}
\def\6#1{{\mathcal{#1}}}

\newcommand{\de}{\partial}
\newcommand{\R}{\mathbb R}
\newcommand{\Ha}{\mathbb H}
\newcommand{\al}{\alpha}

\newcommand{\tv}{\widetilde{\varphi}}

\newcommand{\Z}{\mathbb Z}
\newcommand{\C}{\mathbb C}

\newcommand{\Aut}{{\sf Aut}(\mathbb D)}
\newcommand{\D}{\mathbb D}
\newcommand{\oD}{\overline{\mathbb D}}

\newcommand{\N}{\mathbb N}

\def\Re{{\sf Re}\,}
\def\Im{{\sf Im}\,}

\newcommand{\strip}{\mathbb{S}}

\newcommand{\UD}{\mathbb{D}}

\newcommand{\mcite}[1]{\csname b@#1\endcsname}

\theoremstyle{theorem}

\def\Aut{{\sf Aut}}

\def\Re{{\sf Re}\,}
\def\Im{{\sf Im}\,}

\newtheorem{theorem}{Theorem}[section]
\newtheorem{lemma}[theorem]{Lemma}
\newtheorem{proposition}[theorem]{Proposition}
\newtheorem{corollary}[theorem]{Corollary}

\theoremstyle{definition}
\newtheorem{definition}[theorem]{Definition}
\newtheorem{example}[theorem]{Example}

\theoremstyle{remark}
\newtheorem{remark}[theorem]{Remark}

\numberwithin{equation}{section}

\title[Topological invariants]{Topological invariants for semigroups of holomorphic self-maps of the unit disc}

\author[F. Bracci]{Filippo Bracci$^\dag$}
\address{F. Bracci: Dipartimento di Matematica, Universit\`a di Roma ``Tor Vergata", Via della Ricerca
Scientifica 1, 00133, Roma, Italia.} \email{fbracci@mat.uniroma2.it}

\author[M. D. Contreras]{Manuel D. Contreras$^\ddag$}

\author[S. D\'{\i}az-Madrigal]{Santiago D\'{\i}az-Madrigal$^\ddag$}
\address{M. D. Contreras, S. D\'{\i}az-Madrigal: Camino de los Descubrimientos, s/n\\
Departamento de Matem\'{a}tica Aplicada~II and IMUS\\ Universidad de Sevilla\\ Sevilla,
41092\\ Spain.}\email{contreras@us.es} \email{madrigal@us.es}

\subjclass[2010]{Primary 37C10, 30C35; Secondary 30D05, 30C80, 37F99, 37C25}
\keywords{Semigroups of holomorphic functions; topological invariants; topological normal forms}

\thanks{$^\dag\,$Partially supported by the ERC grant ``HEVO - Holomorphic Evolution Equations'' n. 277691.}

\thanks{$^\ddag$ Partially supported by the \textit{Ministerio
de Econom\'{\i}a y Competitividad} and the European Union (FEDER), project MTM2012-37436-C02-01, and  by \textit{La Consejer\'{\i}a de Educaci\'{o}n y Ciencia de la Junta de Andaluc\'{\i}a}.}

\long\def\REM#1{\relax}

\begin{document}
\maketitle

\selectlanguage{french}
\begin{abstract}
Soient $(\varphi_t)$, $(\phi_t)$ deux semi-groupes \`a un param\`etre d'endomorphismes holomorphes du disque unit\'e $\mathbb D\subset \mathbb C$. Soit $f:\mathbb D \to \mathbb D$ un hom\'eomorphisme. Nous montrons que si $f \circ \phi_t=\varphi_t \circ f$ pour tout $t\geq 0$, alors $f$ s'\'etend \`a un hom\'eomorphisme de $\overline{\mathbb D}$ en dehors des arcs de contact exceptionnels maximaux (en particulier, si l'on consid\`ere des semi-groupes elliptiques, $f$ s'\'etend toujours \`a un homéomorphisme de $\overline{\mathbb D}$). En utilisant ce r\'esultat, nous \'etudions les invariants topologiques pour les semi-groupes \`a un param\`etre d'endomorphismes holomorphes du disque unit\'e. 
\end{abstract}

\selectlanguage{english}
\begin{abstract}
Let $(\varphi_t)$, $(\phi_t)$ be two one-parameter semigroups of holomorphic self-maps of the unit disc $\mathbb D\subset \mathbb C$. Let $f:\mathbb D \to \mathbb D$ be a homeomorphism. We prove that, if $f \circ \phi_t=\varphi_t \circ f$ for all $t\geq 0$, then $f$ extends to a homeomorphism of $\overline{\mathbb D}$ outside exceptional maximal contact arcs (in particular, for elliptic semigroups, $f$ extends to a homeomorphism of $\overline{\mathbb D}$).  Using this result, we study topological invariants for one-parameter semigroups of holomorphic self-maps of the unit disc.
\end{abstract}

\tableofcontents

\section{Introduction}

One-parameter continuous semigroups of holomorphic self-maps  of $\D$ (sometimes just called holomorphic semigroups in $\D$ for short) have been widely studied, see, {\sl e.g.}, \cite{Berkson-Porta, Abate, Shb}. In recent years, particular attention was paid to the boundary behavior of semigroups in $\D$, their boundary trajectories, boundary singularities, and fine properties of semigroups have been discovered and investigated via the associated infinitesimal generators and K\oe nigs functions, see, {\sl e.g.}, \cite{BCD,CD, CD2, CDP, BrGu, P, ES, EST, EKRS, ESZ} and references therein. For instance, the dynamical  type (elliptic, hyperbolic, parabolic), the hyperbolic step, the Denjoy\,--\,Wolff point, boundary (regular or super-repelling) fixed points, regular poles, fractional singularities, maximal contact arcs, can be well understood in terms of the geometry of the image of the  K\oe nigs function and the analytic properties of the infinitesimal generator. 

All these objects are {\sl holomorphic invariants} of a semigroup of holomorphic self-maps. Namely, if two semigroups of holomorphic self-maps  are holomorphically conjugated  through an automorphism of the unit disc, there is a one-to-one correspondence among the objects listed before and the way they are displaced along the boundary of the disc. This follows easily from the fact that automorphisms of the unit disc are linear fractional maps. Therefore, holomorphic invariants form a huge family and each class of holomorphic conjugations is relatively small.  One might expect that lowering the regularity of the conjugation map, the number of invariants decreases. 

In this paper we are interested in studying {\sl topological invariants} of semigroups of holomorphic self-maps of the unit disc. Namely, we consider properties of holomorphic semigroups which are invariant under conjugation via homeomorphisms of the unit disc, {\sl without any assumption on the regularity of the conjugacy map on the boundary of the unit disc}. One might expect that all holomorphic invariants related to the boundary behavior are destroyed. However, and quite surprisingly, most of them survive and are topological invariants. Roughly speaking, the holomorphic invariants which are related to the isometric (with respect to the Poincar\'e metric) nature of holomorphic conjugacies are destroyed under topological conjugation, but, those invariants which are related to the dynamics survive, with the exception of some special maximal contact arcs, which we call {\sl exceptional maximal arcs}.

To be more precise, let us fix some notations. We refer the reader to the next sections for the corresponding definitions. Let $(\v_t)$   be a (continuous $1$-parameter) semigroup of holomorphic self-maps of the unit disc $\D$. If $(\v_t)$ is not a group of elliptic automorphisms, there exists a unique point $x\in \oD$, called the Denjoy\,--\,Wolff point of $(\v_t)$, such that $\lim_{t\to\infty}\v_t(z)=x$ for all $z\in \D$. The semigroup $(\v_t)$ is called {\sl elliptic} if its Denjoy\,--\,Wolff point belongs to $\D$. If $(\v_t)$ is non-elliptic, it is called {\sl parabolic}  if for every $t\geq 0$ the non-tangential limit of $\v_t'$ at its Denjoy\,--\,Wolff point is $1$, otherwise it is called {\sl hyperbolic}. 

While it is clear that elliptic and non-elliptic semigroups of holomorphic self-maps of $\D$ cannot be topologically conjugated, we first prove in Corollary \ref{conjugatohip} that every non-elliptic semigroup of holomorphic self-maps of $\D$ is topologically conjugated to a hyperbolic  semigroup of holomorphic self-maps of $\D$. 

The basic tool for this result and the next ones, is provided by the existence of universal  holomorphic models for semigroups of holomorphic self-maps, as introduced in \cite{Cowen, BrAr}, and the characterization of topological conjugation between two semigroups via topological conjugation of their holomorphic models (see Section \ref{model}). 

In \cite{P}, P. Gumenyuk proved that for every $x\in \de\D$ and for every $t\geq 0$, $\v_t$ has non-tangential limit at $x$, so that $\v_t(x)$ is well-defined, and the curve $[0,+\infty]\ni t\mapsto \v_t(x)\in \oD$ is continuous. Therefore one can define the {\sl life-time} $T(x)$ of $x\in \de \D$ on the boundary to be the supremum of $t\in [0,+\infty]$ such that $\v_t(x)\in \de \D$ (see Section \ref{max-tra}). If $T(x)>0$, it follows from \cite{BrGu} that either $x$ is a fixed point of the semigroup (and in such a case $T(x)=\infty$) or $x$ belongs to a {\sl maximal contact arc}, that is, $A$ is a non empty open arc in $\de \D$ maximal with respect to the property that the infinitesimal generator of $(\v_t)$ extends holomorphically through $A$, never vanishes and it is tangent to $\de A$.  If $A$ is a maximal contact arc and $T(x)=\infty$, then we call $A$ an {\sl exceptional maximal contact arc}. In this case, the ending point of $A$ (with respect to the natural orientation given by the curve $t\mapsto \v_t(x)$), is the Denjoy\,--\,Wolff point of $(\v_t)$. In particular, if $(\v_t)$ is elliptic, it does not admit exceptional maximal contact arcs. A semigroup has at most two exceptional maximal contact arcs. 

In Proposition \ref{time-cont-max}, we prove that the life-time  is a continuous function on each maximal contact arc and we use it to characterize the starting point of a maximal contact arc. 

In case $x\in\de\D$ is a fixed point for $(\v_t)$ which is super-repelling (or non-regular) we show, as a consequence of  Proposition \ref{compon-bound}, that  $x$ is either the starting point of a maximal contact arc or there exists a backward orbit of $(\v_t)$ landing at $x$.  

Let us denote by $M(\v_t)\subset \de\D$  the set of points which belong to an exceptional maximal contact arc for $(\v_t)$  (note that $M(\v_t)$ is open and possibly empty). Let $x\in \oD$ be the Denjoy\,--\,Wolff point of $(\v_t)$ and let
\[
E(\v_t):=\overline{M(\v_t)\cup \{x\}}\cap \de \D.
\]
Note that $E(\v_t)=\emptyset$ if $x\in \D$ while $\{x\}\subseteq E(\v_t)$ if $x\in \de \D$.
We can now state the main result of this paper:
 
\begin{theorem}\label{main}
Let $(\v_t), (\phi_t)$ be two semigroups of holomorphic self-maps of $\D$, which are not elliptic groups. Let $x\in \oD$ be the Denjoy\,--\,Wolff point of $(\v_t)$ and $y\in \oD$  the Denjoy\,--\,Wolff point of $(\phi_t)$. Assume that $f:\D \to \D$ is a homeomorphism (no regularity on $\de\D$ is assumed) such that $f \circ \phi_t =\v_t\circ f$ for all $t\geq 0$. Then $f$ extends to a homeomorphism from $\oD\setminus E(\phi_t)$ to $\oD\setminus E(\v_t)$.

 Moreover, for all  $p\in \de \D\setminus E(\phi_t)$ it holds $T(p)=T(f(p))$ and 
 \[
 f(\phi_t(p))=\v_t(f(p))\quad \hbox{for all $t\geq 0$}.
 \]
  
Finally,  if $x\in \de \D$, then $y\in \de \D$ and $\angle\lim_{z\to x}f(z)=y$.
\end{theorem}

The previous theorem for the non-elliptic case is the content of Proposition \ref{estensionebuona} and Proposition \ref{inizio-rego}. The way to adapt the proofs for the elliptic case is  sketched  in Section \ref{passtoelliptic}. The proof is based on the combined use of the universal holomorphic model and Carath\'eodory's prime ends topology.

With Theorem \ref{main} at hands, it is quite easy to prove that most of the objects described before are topological invariants. These are described in Section \ref{tinvariants}. We summarize here such results:

\begin{proposition}
The following are topological invariants for semigroups of holomorphic self-maps of $\D$:
\begin{enumerate}
\item fixed points in $\D$,
\item boundary regular fixed points which do not start exceptional maximal contact arcs,
\item super-repelling fixed points which start  maximal, non exceptional, contact arcs,
\item super-repelling fixed points  having  backward orbits of $(\v_t)$ landing at such points,
\item maximal contact arcs which are not exceptional,
\item life time on the boundary of a boundary point,
\item boundary points of continuity for a semigroup, {\sl i.e.}, points on $\de \D$ such that every element of the semigroup has unrestricted limits at such points.
\end{enumerate}
\end{proposition}

In Section \ref{exceptionalb} we study the behavior of topological conjugations on exceptional maximal contact arcs and the Denjoy\,--\,Wolff point. In particular, in Proposition \ref{inizio-rego} we show that the topological conjugacy map always has the non-tangential limit at a boundary regular fixed point which is the initial point of an exceptional maximal contact arc and such a limit is  a boundary regular fixed point (possibly the Denjoy\,--\,Wolff point). However, as shown with several examples, the unrestricted limits at such points and at the Denjoy\,--\,Wolff point  might fail to exist, and exceptional maximal contact arcs can be mapped entirely to the Denjoy\,--\,Wolff point.  

Finally,  in Section \ref{passtoelliptic}, we sketch how to extend these results to the case of  elliptic semigroups.

\section{Preliminaries}

\subsection{Contact and fixed  points}
For the unproven statements, we refer the reader to, {\sl e.g.}, \cite{Abate}, \cite{CMbook} or \cite{Shb}.

Let $f:\D \to \D$ be holomorphic, $x\in \de \D$, and let
\[
\al_x(f):=\liminf_{z\to x} \frac{1-|f(z)|}{1-|z|}.
\]
From Julia's lemma it follows that $\al_x(f)>0$. The number $\al_x(f)$ is called the {\sl boundary dilatation coefficient} of $f$.

If $f:\D \to \C$ is a map and $x\in \de\D$, we write $\angle\lim_{z\to x}f(z)$ for the non-tangential (or angular) limit of $f$ at $x$. If this limit exists and no confusion arises, we  denote its value simply by~$f(x)$.

\begin{definition}
Let $f:\D \to \D$ be holomorphic. A point $x\in \de \D$ is said to be a \textsl{contact point} (respectively, \textsl{boundary fixed point}) of~$f$, if $f(x):=\angle \lim_{z\to x}f(z)$ exists and belongs to~$\de\D$ (respectively, coincides with~$x$). If in addition
\begin{equation}\label{EQ_reg-contact-point}
\al_f(x)<+\infty,
\end{equation}
then $x$ is called a \textsl{regular contact point} (respectively, \textsl{boundary regular fixed point}) of $f$.

A boundary fixed point which is not regular is said to be \textsl{super-repelling}.
\end{definition}
\begin{remark}
By the Julia\,--\,Wolff\,--\,Carath\'eodory theorem, condition~\eqref{EQ_reg-contact-point} in the above definition is sufficient on its own for~$x\in\partial \D$ to be a regular contact point of~$f$.
\end{remark}

If $f:\D\to\D$ is holomorphic, neither the identity nor an elliptic automorphism, by the Denjoy\,--\,Wolff theorem, there exists a unique point $x\in\oD$, called the {\sl Denjoy\,--\,Wolff point} of $f$, such that $f(x)=x$ and the sequence of iterates $\{f^{\circ k}\}$ converges uniformly on compacta of $\D$ to the constant map $z\mapsto x$. Moreover, if $x\in \de\D$ then $x$ is a boundary regular fixed point of $f$ and $\al_f(x)\in (0,1]$.

\subsection{Semigroups}

A (one-parameter) semigroup $(\phi_t)$ of holomorphic self-maps of~$\UD$ is a continuous homomorphism $t\mapsto \phi_t$ from the
additive semigroup $(\R_{\ge0}, +)$ of non-negative real numbers to the
semigroup $({\sf Hol}(\D,\D),\circ)$ of  holomorphic self-maps
of $\D$ with respect to composition, endowed with the
topology of uniform convergence on compacta.

It is known that if $(\phi_t)$ is a semigroup of holomorphic self-maps of~$\UD$ and ${\phi_{t_0}}$ is an automorphism of $\D$ for some~$t_0>0$, then $(\phi_t)$ can be extended to a  group in~$\Aut(\D)$.

\begin{definition}
A {\sl boundary (regular) fixed point}  for a  semigroup
$(\phi_t)$ of holomorphic self-maps of $\D$ is a point $x\in \de \D$ which is a boundary (regular) fixed point of~$\phi_t$ for all~$t>0$.
\end{definition}

Let $(\phi_t)$ be a  semigroup
 of holomorphic self-maps of $\D$. It is well known that  $\phi_{t_0}$ has a fixed point in $\D$ for some $t_0>0$ if and only if there exists $x\in \D$ such that $\phi_t(x)=x$  for all $t\geq 0$. In such a case, the semigroup is called {\sl elliptic}. Moreover, there exists $\lambda\in \C$ with $\Re \lambda\leq 0$ such that $\phi_t'(x)=e^{\lambda t}$ for all $t\geq 0$. The elliptic semigroup $(\phi_t)$ is a group if and only if $\Re \lambda=0$. The number $\lambda$ is called the {\sl spectral value} of the elliptic semigroup.
 
If the semigroup $(\phi_t)$ is not elliptic, then there exists a unique $x\in \de \D$ which is the Denjoy\,--\,Wolff point of $\phi_t$ for all $t>0$.  Moreover, there exists $\lambda\leq 0$, the {\sl dilation of $(\phi_t)$}, such that
\[
\alpha_{\phi_t}(x)=e^{\lambda t}\quad t\geq 0.
\]
A non-elliptic semigroup is said {\sl hyperbolic} if its dilation is non-zero, while it is said {\sl parabolic} if the dilation is $0$. 
 
It is also known (see, \cite[Theorem 1]{CDP}, \cite[Theorem 2]{CDP2}, \cite[pag. 255]{Siskakis-tesis}, \cite{ES}) that  a point $x\in \de \D$ is a boundary (regular) fixed point of $\phi_{t_0}$ for some $t_0>0$ if and only if it is a boundary (regular) fixed point of $\phi_t$ for all $t\geq 0$.

By Berkson and Porta's theorem \cite[Theorem~(1.1)]{Berkson-Porta}, if $(\phi_t)$
is a  semigroup in of holomorphic self-maps of $\D$, then $t\mapsto \phi_t(z)$
is real-analytic and there exists a unique holomorphic vector field
$G:\D\to \C$ such that
\[
\frac{\de \phi_t(z)}{\de
t}=G(\phi_t(z))\qquad\text{for all $z\in\UD$ and all~$t\ge0$}.
\]
This vector field $G$, called the  {\sl infinitesimal generator} of $(\phi_t)$,  is {\sl semicomplete} in the sense that the Cauchy problem
\[
\begin{cases}
\frac{dx(t)}{dt}=G(x(t)),\\
x(0)=z,
\end{cases}
\]
has a  solution $x^z:[0,+\infty)\to \D$ for every $z\in \D$.
Conversely, any semicomplete holomorphic vector field in $\D$
generates a semigroup of holomorphic self-maps of $\D$.

\section{Topological and holomorphic models for semigroups}\label{model}

\subsection{Semigroups and holomorphic models}

\begin{definition}
Let $(\phi_t)$ be a semigroup of holomorphic self-maps of $\D$. A {\sl topological model} for $(\phi_t)$ is a triple $(\Omega, h, \Phi_t)$ such that $\Omega$ is an open subset of $\C$, $\Phi_t$ is a group of (holomorphic) automorphisms of $\Omega$ and $h: \D \to h(\D)\subset \Omega$ is a  homeomorphism on the image (that is, it is open, continuous and injective), $h\circ \phi_t= \Phi_t\circ h$ and 
\begin{equation}\label{absorbing}
\cup_{t\geq 0} \Phi_t^{-1}(h(\D))=\Omega.
\end{equation}
A {\sl holomorphic model}  is a topological model $(\Omega, h, \Phi_t)$ such that $h: \D \to \Omega$ is univalent.
\end{definition}
The previous notion of holomorphic model was introduced in \cite{BrAr}, where it was proved that every semigroup of holomorphic self-maps of any complex manifold admits a  holomorphic model, unique up to ``holomorphic equivalence''. Moreover, a model is ``universal'' in the sense that every other conjugation of the semigroup to a group of automorphisms factorize through the model (see \cite[Section 6]{BrAr} for more details).

Notice that given a topological model $(\Omega, h, \Phi_t)$ for a semigroup $(\phi_t)$ of holomorphic self-maps of $\D$ then $(\phi_t)$ is a group if and only if $h(\D)=\Omega$.

\begin{definition}
Let $(\phi_t)$ and $(\tilde{\phi}_t)$ be two semigroups of holomorphic self-maps of $\D$. Let  $(\Omega, h, \Phi_t)$ be a holomorphic (respectively, topological) model for $(\phi_t)$, and $(\tilde{\Omega}, \tilde{h}, \tilde{\Phi}_t)$ a holomorphic (respectively, topological) model for $(\tilde{\phi}_t)$.  We say that $(\Omega, h, \Phi_t)$ and $(\tilde{\Omega}, \tilde{h}, \tilde{\Phi}_t)$ are {\sl holomorphically (respect., topologically) conjugated} if there exists a biholomorphism (respect., homeomorphism) $\tau: \Omega\to \tilde{\Omega}$ such that $\tau \circ \Phi_t=\tilde{\Phi}_t \circ \tau$ and $\tau(h(\D))=\tilde{h}(\D)$. The map $\tau$ is called a {\sl holomorphic (respect., topological) conjugation of models}.

In case $(\phi_t)=(\tilde{\phi}_t)$ and $\tilde{h}=\tau \circ h$, we say that the models $(\Omega, h, \Phi_t)$  and $(\tilde{\Omega}, \tilde{h}, \tilde{\Phi}_t)$ are holomorphically (respect., topologically) {\sl equivalent}.
\end{definition}

Conjugated semigroups correspond to conjugated models:

\begin{proposition}\label{conj-model}
Let  $(\phi_t), (\tilde{\phi}_t)$ be two semigroups of holomorphic self-maps of $\D$, with  holomorphic (respectively, topological) models $(\Omega, h, \Phi_t)$ and $(\tilde{\Omega}, \tilde{h}, \tilde{\Phi}_t)$. The following are equivalent:
\begin{enumerate}
\item the holomorphic (respectively, topological) models $(\Omega, h, \Phi_t)$ and $(\tilde{\Omega}, \tilde{h}, \tilde{\Phi}_t)$  are holomorphically (respect., topologically) conjugated;
\item the semigroups $(\phi_t)$ and $(\tilde{\phi}_t)$ are holomorphically (respect., topologically) conjugated, that is, there exists an automorphism (respect., homeomorphism) $T: \D \to \D$ such that $T\circ \tilde{\phi}_t \circ T^{-1}=\phi_t$ for all $t\geq 0$.
\end{enumerate}
\end{proposition}

\begin{proof} (1) implies (2). Let  $\tau$ be a  holomorphic (respect., topological) conjugation of models between $(\Omega, h, \Phi_t)$ and $(\tilde{\Omega}, \tilde{h}, \tilde{\Phi}_t)$, that is, $\tau: \Omega\to \tilde{\Omega}$ is a biholomorphism (respect., homeomorphism) such that $\tau \circ \Phi_t=\tilde{\Phi}_t \circ \tau$ and $\tau(h(\D))=\tilde{h}(\D)$.  Define $T:=\tilde{h}^{-1} \circ \tau \circ h$. It is easy to see that $T$ is a holomorphic (respect., topological) conjugation between $(\phi_t)$ and $(\tilde{\phi}_t)$.

(2) implies (1). Let $T$ be a holomorphic (respect., topological) conjugation between $(\phi_t)$ and $(\tilde{\phi}_t)$. 
Next, we extend $\tau:=h^{-1}\circ T\circ \tilde h$ to a biholomorphism (respect., homeomorphism) $\tau: \Omega\to \tilde{\Omega}$ which intertwines $\Phi_t$ with $\tilde{\Phi}_t$ in the following way. Let $\Omega_t:=\Phi_t^{-1}(h(\D))$ and $\tilde{\Omega}_t:=\tilde{\Phi}_t^{-1}(\tilde{h}(\D))$, $t\geq 0$. Note that, if $t\geq s$ then $\Omega_s\subseteq \Omega_t$, since $\Phi_t=\Phi_{t-s}\circ \Phi_s$ and $h(\D)$ is invariant under $\Phi_t$ for all $t\geq 0$. Moreover, $\Phi_t(\Omega_t)=h(\D)$ and, by the definition of model, $\Omega=\cup_{t\geq 0}\Omega_t$. Similarly for $\tilde{\Omega}$.

For $t\geq 0$, define $\tau_t: \Omega_t\to \tilde{\Omega}_t$ by 
\[
\tau_t:=\tilde{\Phi}_t^{-1} \circ  \tilde{h} \circ T \circ h^{-1} \circ \Phi_t.
\]
Clearly, $\tau_t$ is a biholomorphism (respect., homeomorphism) from $\Omega_t$ and $\tilde{\Omega}_t$. Also, by definition, $\tilde{\Phi}_t \circ \tau_0=\tau_0 \circ \Phi_t|_{\Omega_0}$ for all $t\geq 0$. 

Let $0\leq s\leq t$ and let $\omega\in \Omega_s$. Then
\[
\tau_t(\omega)=\tilde{\Phi}_t^{-1}(\tau_0(\Phi_{t-s}(\Phi_s(\omega))))=\tilde{\Phi}_t^{-1}(\tilde{\Phi}_{t-s}(\tau_0(\Phi_s(\omega))))=\tau_s(\omega).
\]
Therefore, for $0\leq s\leq t$ it holds $\tau_t|_{\Omega_s}=\tau_s$. Hence, the map $\tau:\Omega\to \tilde{\Omega}$ defined by $\tau|_{\Omega_t}:=\tau_t$ is well defined and it is a 
biholomorphism (respect., homeomorphism) from $\Omega$ and $\tilde{\Omega}$.

Finally, by a similar argument as before, one can show that for all $t\geq 0$, $\tau_t \circ \Phi_s|_{\Omega_t}=\tilde{\Phi_s}\circ \tau_t$ for all $s\geq 0$, hence $\tau\circ \Phi_s=\tilde{\Phi}_s\circ \tau$ for all $s\geq 0$, concluding the proof. 
\end{proof}


Holomorphic models always exist and are unique up to holomorphic equivalence of models. In what follows we denote by  $\Ha:=\{\zeta \in \C: \Re \zeta>0\}$, $\Ha^{-}:=\{\zeta\in \C: \Re \zeta<0\}$ and, given $\rho>0$, $\strip_\rho:=\{\zeta\in \C: 0<\Re \zeta<\rho\}$. We simply write $\strip:=\strip_1$. The following result sums up the results in \cite{BrAr, Cowen}, see also \cite{Abate}.

\begin{theorem}\label{modelholo}
Let $(\phi_t)$ be a semigroup of holomorphic self-maps of $\D$. 
\begin{enumerate}
\item If $(\phi_t)$ is a {\sl group} of elliptic automorphism, then it has a holomorphic model $(\D, h, z\mapsto e^{i\theta t}z)$ for some $\theta\in (-\pi,\pi]$, with $h(\D)=\D$.
\item If $(\phi_t)$ is elliptic, not a group, then $(\phi_t)$ has a holomorphic model  $(\C, h, z\mapsto e^{\lambda t} z)$, for some  $\lambda\in \C$ such that $\Re \lambda<0$.
\item If $(\phi_t)$ is non-elliptic, then 
$(\phi_t)$ has a holomorphic model $(\Omega, h, \Phi_t)$ where   $\Omega=\C, \Ha, \Ha^{-},\strip_\rho$, for some $\rho>0$, and  $\Phi_t(z)=z+it$, $t\geq 0$.
\end{enumerate}
\end{theorem}
 The spectral value of an elliptic semigroup is clearly a holomorphic invariant. 

The function $h$ in the previous model is called the {\sl K\oe nigs function} of the semigroup. All the previous models are holomorphically non-equivalent. In fact, we have that
\begin{itemize}
\item If $\Omega=\C$ and $\Phi_t(z)=e^{\lambda t} z$, for some  $\lambda\in \C$ such that $\Re \lambda<0$, then $(\phi_t)$ is a semigroup of {\sl elliptic type}.
\item If $\Omega=\C$ and $\Phi_t(z)=z+it$, then $(\phi_t)$ is a semigroup of {\sl parabolic type of zero hyperbolic step}. 
\item If $\Omega=\Ha, \Ha^-$ and $\Phi_t(z)=z+ it$, then $(\phi_t)$ is a semigroup of {\sl parabolic type of positive hyperbolic step}. 
\item If $\Omega=\strip_\rho$, for some $\rho>0$,  and $\Phi_t(z)=z+ it$, then $(\phi_t)$ is a semigroup of {\sl hyperbolic type}.
\end{itemize}

In any of the non-elliptic cases, the model is given by  $(\Omega=I\times \R,h,z+it)$, where $I$ is any of the intervals $(-\infty,0)$, $(0,+\infty)$, $(0,\rho)$, for some  $\rho>0$, or $\R$. This interval $I$ is  a holomorphic invariant. In particular, in the hyperbolic case, $\rho$ depends on the dilation $\lambda$ of the semigroup.

\subsection{Semigroups and topological models}

From a topological point of view, given a semigroup $(\phi_t)$, which is not a group, there are only two possible models:

\begin{proposition}\label{top1}
Let $(\phi_t)$ be a semigroup of holomorphic self-maps of $\D$.
\begin{enumerate}
\item If $(\phi_t)$ is elliptic, not a group, then $(\phi_t)$ has a topological model given by $(\C, h, z\mapsto e^{-t}z)$.
\item If $(\phi_t)$ is non-elliptic, then $(\phi_t)$ has a topological model of  hyperbolic type, given by a model $(\strip, h, z\mapsto z+ it)$.
\end{enumerate}
\end{proposition}
\begin{proof}
Assume that $(\phi_t)$ is parabolic. Let $(\Omega, h, z\mapsto z+it)$ be the model of $(\phi_t)$ given by Theorem \ref{modelholo}. If $\Omega=\C$, define $\v: \C \to \strip$ by $\v(x+iy)=\theta(x)+iy$, where $\theta: \R \to (0,1)$ is any homeomorphism. Then clearly $\v (z+it)=\v(z)+it$ for all $t\in \R$. Now, let $\tv:=\v\circ h$. It is then easy to see that $(\strip, \tv, \Phi_t)$ is a topological model for $(\phi_t)$. In case $\Omega=\Ha$ and $\Phi_t(z)=z+it$, it is enough to replace $\theta$ with any homeomorphism $\theta: (0,\infty)\to (0,1)$. While, if $\Omega=\Ha^-$ one can replace $\theta$ with any homeomorphism $\theta: (-\infty,0)\to (0,1)$. If $\Omega=\strip_\rho$, just define $\varphi(x+iy)=x/\rho+iy$.

In case $(\phi_t)$ is elliptic, not a group, let $(\C, h, z\mapsto e^{\lambda t}z)$ be the model of $(\phi_t)$ given by Theorem \ref{modelholo}, with $\lambda=a+ib$, $a<0$ and $b\in \R$. Define 
\[
\v(\rho e^{i\theta}):=\exp\left(-\left( \frac{1}{a}+i\frac{b}{a}\right)\log \rho \right)e^{i\theta}, \quad \rho\neq 0,
\]
and $\v(0)=0$. It is easy to see that $\v: \C \to \C$ is a homeomorphism and that $\v(e^{\lambda t}w)=e^{-t}\v(w)$ for all $t\in \R$ and $w\in \C$. Then  $(\C, \tv, z\mapsto e^{-t}z)$ is the topological model of $(\phi_t)$, where $\tilde \varphi =\varphi \circ h$. 
\end{proof}

As a consequence we have the following straightforward result:
\begin{corollary}\label{conjugatohip} 
Every non-elliptic semigroup of holomorphic self-maps of $\D$ is topologically conjugated to a hyperbolic  semigroup of holomorphic self-maps of $\D$.
\end{corollary}

Clearly, elliptic and non-elliptic models cannot be topologically equivalent. For groups of automorphisms, there are  more possible topologically inequivalent models: 

\begin{corollary}
Let $(\phi_t)$ be a group of automorphisms of $\D$. 
\begin{enumerate}
\item If $(\phi_t)$ is elliptic with spectral value $i\theta$, with $\theta\in (-\pi,\pi]$, then its topological model is given by $(\D, h, z\mapsto e^{it|\theta|}z)$, with $h(\D)=\D$, 
\item if $(\phi_t)$ is non-elliptic then its topological model is given by $(\strip, h, z\mapsto z+it)$,  with $h(\D)=\strip$.
\end{enumerate}
\end{corollary}
\begin{proof}
If $(\phi_t)$ is non-elliptic, then the result follows at once from Proposition \ref{top1}. If $(\phi_t)$ is elliptic, it has the holomorphic model $(\D, h, z\mapsto e^{it\theta}z)$ by Theorem \ref{modelholo}. If $\theta<0$, the map $z\mapsto \overline{z}$ conjugates the holomorphic model to the model $(\D, h, z\mapsto e^{it|\theta|}z)$
\end{proof}

\begin{remark}
Let $(\varphi_t)$ and $(\phi_t)$ two groups of automorphisms with models $(\D, h, z\mapsto e^{it\theta_1}z)$ and $(\D, \tilde h, z\mapsto e^{it\theta_2}z)$, respectively. Then $(\varphi_t)$ and $(\phi_t)$ are topologically conjugated if and only if $|\theta_1|=|\theta_2|$.  Indeed, if $\theta_1=-\theta_2$ the map $\tau(z)=\overline z$ intertwines the two groups. On the other hand, the semigroups are topologically conjugated, there exists a homeomorphism $T:\D\to \D$ such that $T(e^{\theta_1t}z)= e^{\theta_2t}T(z)$ for all $t>0$ and all $z\in \D$. If  $\theta_1=0$ then $\theta_2=0$ as well, and similarly if $\theta_2=0$ then $\theta_1=0$. In case $\theta_1,\theta_2$ are non-zero, let $t=2\pi/\theta_1$. Then $T(z)= e^{\theta_2 2\pi/\theta_1}T(z)$ for all  $z\in \D$. Hence $\theta_2 /\theta_1=:m\in \Z$. Taking $t=2\pi/(m\theta_1)$ we deduce that $T(e^{2\pi/m}z)= T(z)$ for all $z\in \D$. Therefore $m=\pm1$.
\end{remark}

The previous remark follows also from Naishul's theorem (see, {\sl e.g.}, \cite[Theorem 2.29]{Br}), which states that two germs of elliptic holomorphic maps in $\C$ fixing $0$ which are topologically conjugated by an orientation preserving map, must have the same derivative at $0$. The proof of such a result is however much more complicated than the corresponding results for groups.

It is worth pointing out that while the absolute value of the spectral value is a topological invariant for elliptic groups, all elliptic semigroups which do not form a group, are topologically equivalent to an elliptic semigroup with spectral value $-1$.

\section{Maximal contact arcs and boundary fixed points}\label{max-tra}
\subsection{Trajectories}

The following result is due to P. Gumenyuk.

\begin{theorem} \cite[Thm. 3.1, Prop. 3.2]{P}  (see also \cite{CDP}).  \label{pavelone}
Let $(\phi_t)$ be a semigroup of holomorphic self-maps of $\D$. Then for every $\sigma\in \de \D$ the non-tangential limit $\phi_t(\sigma):=\angle\lim_{z\to \sigma}\phi_t(z)$ exists. Moreover, the curve $[0,+\infty)\ni t\mapsto \phi_t(\sigma)$ is continuous.
\end{theorem}

Given $\sigma \in \overline\D$, the curve $[0,+\infty)\ni t\mapsto \phi_t(\sigma)$ is called the {\sl trajectory of the semigroup associated with $\sigma$}. If $\sigma \in \D$, the trajectory associated with $\sigma$ is contained in $\D$ and converges to the Denjoy\,--\,Wolff point of the semigroup, in case $(\phi_t)$ is not an elliptic group. 

In case $\sigma\in \de \D$ and $\phi_t(\sigma)\in \D$ for some $t>0$, then $\phi_s(\sigma)\in \D$ for all $s\geq t$. This motivates the following definition:

\begin{definition}
Let $(\phi_t)$ be a semigroup of holomorphic self-maps of $\D$ and $\sigma \in \partial \D$. The  {\sl life-time} (on the boundary) of the trajectory associated with $\sigma$ is
\[
T(\sigma):=\sup\{t\geq 0: \phi_t(\sigma)\in \de \D\}\in [0,+\infty].
\]
\end{definition}

Note that if $\sigma$ is a boundary fixed point for a semigroup, then its associated trajectory is constant and the life-time $T(\sigma)=\infty$.


\subsection{Maximal contact arcs}
In this subsection we recall the notion of contact arc introduced by the first author and Gumenyuk in \cite{BrGu} and summarize some related results we will use throughout the paper.

\begin{definition}Let $(\phi_t)$ be a semigroup of holomorphic self-maps of $\D$ with associated infinitesimal generator $G$. An open arc $A\subset\de\D$ is said to be a {\sl contact arc} for $(\phi_t)$, if
$G$ extends holomorphically to $A$ with $\Re\{\overline z\,G(z)\}=0$ and $G(z)\neq0$ for all $z\in A$. A contact arc $A$ is {\sl maximal}  if there exists no other contact arc $B$ for $(\phi_t)$ such that $B\supsetneq A$.
Each contact arc~$A$ is endowed with a natural orientation induced by the unit vector field $G/|G|$. This determines  the \textsl{initial} and \textsl{final} end-points of~$A$.
\end{definition}

If $(\phi_t)$ is a semigroup, not an elliptic group, and $A$ is a contact arc for the semigroup, then $A\neq \partial \D$. Notice that for elliptic groups, $\partial \D$ is a contact arc. From now on, we assume that the semigroup is not an elliptic group.

The following results give  a characterization of contact arcs in terms of the K\oe nigs map. 

\begin{theorem}  \cite[Sect. 3]{BrGu}. \label{chac_contact_arc}
Let $(\phi_t)$ be a  semigroup of holomorphic self-maps, not an elliptic group, with associated infinitesimal generator $G$ and  K\oe nigs function $h$. Let $A\subset \partial \D$ be an open arc. Then the following statements are equivalent  
\begin{enumerate}
\item $G$ extends holomorphically to $A$ with $\Re\{\overline z\,G(z)\}=0$ and $G(z)\neq0$ for all $z\in A$;
\item $h$ extends holomorphically through $A$ and 
\begin{itemize}
\item there exists $c\in \overline I$ such that $\Re h(z)=c$ for all $z\in A$, if $(\phi_t)$ is non-elliptic with a holomorphic model $(I\times \R,h,z+it)$;
\item $h(A)$ is contained in a spiral  $\{z\in \C: z=e^{t\lambda}v, t\in \R\}$ for some $v\in \C\setminus\{0\}$, if $(\phi_t)$ is elliptic with spectral value $\lambda$.
\end{itemize}
\end{enumerate}
\end{theorem}

The relationship between contact arcs and trajectories is given by the following result. As a notation, if $A\subset \de \D$ is an open sub-arc with end points $p,q$, we denote it by $(p,q)$.

\begin{theorem} \cite[Sect. 3]{BrGu}\label{arcomass}
Let $(\phi_t)$ be a semigroup of holomorphic self-maps which is not an elliptic group. 
\begin{enumerate}
\item if $A$ is a contact arc for $(\phi_t)$ and $x\in A$, then $x$ is a regular contact point for $\phi_t$ for $t\in (0, T(x))$. Moreover, the map $t\mapsto \phi_t(x)$ is a homeomorphism of $(0,T(x))$ onto  the open sub-arc $(x, \phi_{T(x)}(x))\subseteq A$.
\item If $p\in \de \D$ is a contact point for $\phi_{t_0}$ for some $t_0>0$ and it is not a boundary fixed point,  then there exists a maximal contact arc $A$ for $(\phi_t)$ such that $p\in \overline{A}$. 
\item Let $A$ be a maximal contact arc for $(\phi_t)$ with initial point $x_0$ and final point $x_1$. Then 
\begin{itemize}
\item[(i)] either $x_0$ is a boundary fixed point of $(\phi_t)$ or $x_0$ is a contact point for $\phi_{t_0}$ for some $t_0>0$,
\item[(ii)] either $x_1$ is the Denjoy\,--\,Wolff point of $(\phi_t)$ or $\phi_t(x_1)\in \D$ for all $t>0$.
\end{itemize}
\end{enumerate}
\end{theorem}

\begin{definition}
Let  $(\phi_t)$ be a semigroup of holomorphic self-maps which is not an elliptic group. A maximal contact arc for $(\phi_t)$ whose final point is the  Denjoy\,--\,Wolff point of $(\phi_t)$ is called an {\sl exceptional maximal contact arc}.
\end{definition}
Clearly, a semigroup might have only two exceptional maximal contact arcs, and elliptic semigroups do not have any exceptional maximal contact arc. The life-time of a trajectory is continuous on each maximal contact arc:

\begin{proposition}\label{time-cont-max}
Let $(\phi_t)$ be a semigroup of holomorphic self-maps which is not an elliptic group and $A$ a maximal contact arc. 
\begin{enumerate}
\item if $A$ is a  exceptional, then $T(x)=\infty$ for all $x\in \overline{A}$.
\item if $A$ is not exceptional, then the function
$$
\overline{A}\ni x\mapsto T(x) \in [0,+\infty]
$$
is continuous and strictly decreasing (with respect to the natural orientation of $A$ induced by $(\phi_t)$). Moreover, $T$ is bounded from above on $\overline{A}$ if an only if the initial point of $A$ is not a fixed point of $(\phi_t)$.
\end{enumerate}
\end{proposition}
\begin{proof} (1) follows immediately from Theorem \ref{arcomass}.

(2) We prove the result in case $(\phi_t)$ is non-elliptic with holomorphic model $(I\times \R,h,z+it)$, the elliptic case follows similarly. 

By Theorem \ref{chac_contact_arc}, $h$ extends holomorphically through $A$ and therefore it follows that for all $x\in A$ and $t\geq 0$ it holds $h(\phi_t(x))=h(x)+it$.  Moreover, there exists $c\in \overline I$ such that $\Re h(z)=c$ for all $z\in A$. Let $R:=\sup\{ s\in \R: c+is \in  h(A)\}$.  Note that $R<\infty$, since $A$ is not exceptional. Then
\begin{equation}\label{TR}
T(x)=R-\Im h(x).
\end{equation}
Now let $y\in A$. By Theorem  \ref{chac_contact_arc}, the map $t\mapsto \phi_t(y)$ is an homeomorphism of $(0,T(y))$ onto  the open sub-arc $B:=(y, \phi_{T(y)}(y))$ of $A$. Let $t(x)$ denote its inverse. Hence, for $x\in B$
\[
T(x)=T(\phi_{t(x)}(y))=R-\Im h(\phi_{t(x)}(y))=R-\Im h(y)-t(x),
\]
proving that $T$ is continuous and strictly decreasing in $B$ (with respect to the natural orientation of $A$ induced by $(\phi_t)$). By the arbitrariness of $B$, it follows that $T$ is continuous and strictly decreasing in $A$. 

If $x_1$ is the final point of $A$, since $A$ is not exceptional, $T(x_1)=0$, and the previous argument shows that $T$ is continuous on $A\cup\{x_1\}$ as well. 

Now we examine the initial point $x_0$. There are two cases. Either $x_0$ is a fixed point of $(\phi_t)$ or $x_0$ is a contact (not fixed) point of $(\phi_t)$. In the first case $T(x_0)=\infty$, while, in the second case, according to \cite[Remark 3.4]{BrGu}, $T(x_0)<\infty$. 
Let 
\[
S:=\lim_{A\ni x\to x_0} T(x).
\]
Since $T$ is strictly decreasing, such a limit exists, finite or infinite. 

Assume first that $T(x_0)=\infty$. By \cite[ Proposition 2.11(iii) and Lemma 2.10]{BrGu},  $h(A)=\{c+is: s\in (-\infty, R)\}$. Hence, $S=\infty$ by \eqref{TR}.

Assume next $T(x_0)<\infty$. The curve $[0,\infty)\ni s \mapsto \phi_s(x_0)$ is continuous, and  $T$ is continuous on $A$. Hence, $\lim_{s\to 0}T(\phi_s(x_0))=S$, and, clearly $S\leq T(x_0)$ by the definition of life time. On the other hand, if $S<T(x_0)$, let $t\in (S, T(x_0))$. Taking into account that $T$ is strictly decreasing on $A$, it follows that $\phi_t(\phi_s(x_0))\in \D$ for all $s>0$, hence $\phi_t(x_0)\in \D$, a contradiction.
\end{proof}

For our aims, we need to recall from \cite{BrGu} how the maximal contact arcs are related to the K\oe nigs function of the hyperbolic and starlike elliptic models:

\begin{proposition}  \cite[Sect. 3]{BrGu} \label{arc-koenig}
Let $(\phi_t)$ be a semigroup of holomorphic self-maps of $\D$. Let $(\Omega, h, \Phi_t)$ be a holomorphic model of $(\phi_t)$. 
\begin{enumerate}
\item If $(\phi_t)$ is a starlike elliptic semigroup, not a group, then the initial point of a maximal contact arc $A$ is a fixed point if and only if the image of $h(A)$ is not bounded.
\item If $(\phi_t)$ is a non-elliptic semigroup,  the initial point of a maximal contact arc $A$ is a fixed point if and only if $\Im h(A)$ is not bounded from below. Moreover, if  $\Omega=I\times \R$, then $A$ is exceptional  if and only if either $I=(c,a)$ with $c\in \R, a\in (c,+\infty]$    or $I=(a,c)$  with $c\in \R$ and $a\in [-\infty, c)$, and there exists $d\in [-\infty, +\infty)$ such that  $\{c+it: t>d\}= h(A)$.
\end{enumerate}
\end{proposition}

\subsection{Boundary fixed points}

Let $(\phi_t)$ be a semigroup of holomorphic self-maps of $\D$. Assume $(\phi_t)$ is either non-elliptic  with holomorphic model $(\Omega, h, z\mapsto z+it)$ or  starlike elliptic with the model $(\C, h, z\mapsto e^{-t}z)$. Let  $Q:=h(\D)$. Set
\begin{equation}\label{setA}
A(\phi_t):=\begin{cases}
\cap_{t\geq 0} (e^{- t}Q) & \hbox{ if $(\phi_t)$ is starlike elliptic},\\
\cap_{t\geq 0} (Q+it) & \hbox{ if $(\phi_t)$ is non-elliptic}.
\end{cases}
\end{equation}

Following ideas from \cite[Thm. 2.5]{CD}, it is possible to prove that, given a non-elliptic semigroup, if for some $c\in \R$ the line $\{\zeta\in \C: \Re \zeta=c\}$ is contained in $A(\phi_t)$ then $p:=\lim_{t\to -\infty} h^{-1}(c+it)$ is a boundary fixed point of the semigroup. 
Next result shows a characterization of those boundary fixed points that can be reached in this way.
 
\begin{proposition}\label{compon-bound}
Let $(\phi_t)$ be a non-elliptic semigroup of holomorphic self-maps of $\D$. Let $C$ be a connected components of $A(\phi_t)$. Then  $C=C+it$, for all $t>0$, and  there exist $ -\infty\leq a \leq b\leq +\infty$ such that $\overline C=\{\zeta\in \C: a\leq\Re \zeta\leq b\}$.

Moreover, there is a one-to-one correspondence between boundary regular fixed points of $(\phi_t)$, different from the Denjoy\,--\,Wolff point, and connected components of $A(\phi_t)$ for which  there exist $-\infty< a<b<+\infty$ such that $\overline C=\{\zeta\in \C: a\leq\Re \zeta\leq b\}$ associating to $C$ the point $p(C):=\lim_{t\to -\infty} h^{-1}(c+it)$, where $c$ is any point in $(a,b)$, and $p(C)\in \de \D$ is a boundary regular fixed point of $(\phi_t)$.

Also, there is a one-to-one correspondence between boundary super-repelling fixed points of $(\phi_t)$ which are not inital points of a maximal contact arc of $(\phi_t)$ and connected components of $A(\phi_t)$ given by a line associating to  the connected component $C=\{\zeta\in \C: \Re \zeta=c\}$ the point $p(C):=\lim_{t\to -\infty} h^{-1}(c+it)$,  and $p(C)\in \de \D$ is a boundary super-repelling fixed point of $(\phi_t)$.
\end{proposition}
\begin{proof}
 The case of boundary regular fixed points  is in \cite[Thm. 2.5]{CD}. 
 
By \cite[Lemma 4.3]{BrGu}, if $p$ is a super-repelling boundary fixed point of $(\phi_t)$ for which there is no connected component of $A(\phi_t)$ of the form $ C=\{\zeta\in \C: \Re \zeta=c\}$ with $p:=\lim_{t\to -\infty} h^{-1}(c+it)$ then $p$ is the initial point of a maximal contact arc. 
 Conversely, assume by contradiction that $p$ is a boundary super-repelling fixed point which is the initial point of a maximal contact arc $A$ and there exists $c\in \R$ such that  
$p:=\lim_{t\to -\infty} h^{-1}(c+it)$.  

Take a sequence $(t_n)$ converging to $-\infty$ such that the connected component of $h(\D)\cup (\R+it_n)$ containing $c+it_n$ is of the form $(c-\alpha_n,c+\beta_n)$, where $0\leq \alpha_{n+1}< \alpha_n$ and $0\leq \beta_{n+1}< \beta_n$ for all $n\in \N$. Since $p$ is super-repelling, the sequences $(\alpha_n)$ and $(\beta _n)$ converge to zero. Consider the Jordan arc $C_n=h^{-1}((c-\alpha_n,c+\beta_n))$. Since $h$ has no Koebe arcs, the diameter of $(C_n)$ tends to zero. Therefore, for $n$ large enough, one of the end points of the Jordan arc $C_n$ belongs to $A$, say this end point is $\lim_{C_n\ni z\to c+\beta_n} h^{-1}(z)$. By Theorem \ref{chac_contact_arc}(2), $\Re h(z) =c+\beta_n=c+\beta_{n+1}$ for all $z\in A$, contradicting $\beta_{n+1}< \beta_n$.
\end{proof}

A similar statement holds for the starlike elliptic case, replacing open strips $\{z\in \C: a\leq\Re z\leq b\}$ with angles $\{z\in \C: a\leq \arg z\leq b\}$ and vertical lines with half-lines $\{z=tv: t\geq 0\}$, $v\in \C\setminus\{0\}$. 

\begin{definition}
Let $(\phi_t)$ be a semigroup of holomorphic self-maps of $\D$. A boundary super-repelling fixed point $p\in \de \D$ for $(\phi_t)$ is called
\begin{enumerate}
\item a boundary super-repelling fixed point {\sl of first type} if  it corresponds to a connected component of $A(\phi_t)$ which is a line,
\item a boundary super-repelling fixed point {\sl of second type} if  it is the initial point of a maximal contact arc for $(\phi_t)$ which is not exceptional. \item a boundary super-repelling fixed point {\sl of third type} if  it is the initial point of an exceptional  maximal contact arc for $(\phi_t)$.
\end{enumerate}
\end{definition}

Note that in case (1) there exists a backward orbit of $(\phi_t)$ which lands at $p$. Indeed, if $p$ corresponds to the connected component $C=\{\zeta\in \C: \Re \zeta=c\}$ of $A(\phi_t)$, then the curve $[0,+\infty]\ni t\mapsto h^{-1}(c+it)\in \D$ is a backward orbit of $(\phi_t)$ landing at $p$.

By \cite[Lemma 4.3]{BrGu}, a super-repelling fixed point is necessarily one of the above types. In particular, we have the following result:

\begin{corollary}\label{super-rep-divide}
Let $(\phi_t)$ be a semigroup of holomorphic self-maps of $\D$. A boundary super-repelling fixed point $p\in \de \D$ for $(\phi_t)$ is  either  the starting point of a maximal contact arc or there exists a backward orbit of $(\phi_t)$ landing at $p$.
\end{corollary}

\begin{remark}
If $(\phi_t)$ is elliptic, then there are no boundary super-repelling fixed points of third type, while, if $(\phi_t)$ is non-elliptic then there are at most two boundary super-repelling fixed points of third type.
\end{remark}

Next result is due to P. Gumenyuk and shows the relationship between the boundary behavior of K\oe nigs functions and boundary fixed  points.

\begin{theorem}  \cite[Prop. 3.4, Prop. 3.7]{P} \label{pavel}
Let $(\phi_t)$ be a non-elliptic  semigroup of holomorphic self-maps of $\D$ with holomorphic model $(\Omega, h, z\mapsto z+it)$. Then
\begin{enumerate}
\item for every $\sigma\in \de \D$ the non-tangential limit $\angle\lim_{z\to \sigma} h(z)\in \overline{\C}$ exists.
\item Let $\sigma\in \de \D$. The (unrestricted) limit $\lim_{z\to \sigma}\Re h(z)$ exists finitely if and only if $\sigma$ is not a regular fixed point for $(\phi_t)$.
\item $\sigma\in \de \D$ is a boundary fixed point for $(\phi_t)$ different from the Denjoy\,--\,Wolff point if and only if $\lim_{z\to \sigma} \Im h(z)=-\infty$.
\item If $\sigma\in \de \D$ is not the Denjoy\,--\,Wolff of $(\phi_t)$ then $\limsup_{z\to \sigma} \Im h(z)<+\infty$.
\end{enumerate}
\end{theorem}

A similar result (see \cite{P}) holds for the starlike elliptic case replacing the real and imaginary part of $h$ with the modulus and the argument of $h$.

In the sequel we will need also the following characterization of points belonging to exceptional maximal contact arcs in terms of Carath\'eodory's prime end theory (see, {\sl e.g.} \cite[Chapter 9]{CL}, \cite[Sections 2.4-2.5]{Pommerenke2}). As a matter of notation, if $h:\D \to \C$ is univalent, we denote by  $\hat{h}$ the homeomorphism induced by $h$ from $\de \D$ and the space of Carath\'eodory's prime ends of $h(\D)$. 

\begin{proposition}\label{prime-exc}
Let $(\phi_t)$ be a non-elliptic semigroup of holomorphic self-maps of $\D$. Let $(I\times \R, h, z\mapsto z+it)$ be the holomorphic model of $(\phi_t)$, where $I=\R, (0,+\infty), (-\infty,0)$ or $(0,\rho)$ for some $\rho>0$. Let $M$ be the union of the exceptional maximal contact arcs for  $(\phi_t)$ (possibly $M=\emptyset$). Let $p\in\de\D$ be different from the Denjoy\,--\,Wolff point of $(\phi_t)$. Then $p\not\in \overline{M}$ if and only if the following condition holds:
\begin{itemize}
\item[(L)] there exists a null chain $(C_n)$ in $h(\D)$ representing  $\hat{h}(p)$ such that there exist a compact subinterval $I'\subset I$ and $N\in \N$ such that the connected component $V_n$ of $h(\D)\setminus C_n$ which does not contain $C_0$ satisfies $V_n\subset I'\times \R$ for all $n\geq N$.
\end{itemize}
\end{proposition}

\begin{proof}
Assume $\hat M$ is an exceptional maximal contact arc for $(\phi_t)$. By Proposition \ref{arc-koenig}, $I\neq \R$ and we can assume without loss of generality that $I=(0,\rho)$, with $\rho\in (0,+\infty]$ and that $h(\hat M)=\{it, t>d\}$ for some $d\in[-\infty,+\infty)$. 

Let $p\in \de \D$ be a point different from the Denjoy\,--\,Wolff point of $(\phi_t)$.

If $p\in \hat M$, since by Theorem \ref{chac_contact_arc}, $h$ extends holomorphically through $p$ and $h(p)\in  \{it, t>d\}$,  every null chain $(C_n)$ representing $\hat{h}(p)$ converges to $\{h(p)\}$, hence it is not possible to find a closed subinterval $I'\subset I$ and $N\in \N$ such that  $V_n\subset I'\times \R$ for all $n\geq N$. 

If $p$ is the initial point of $\hat M$, assume by contradiction that $(C_n)$ is a null chain representing $\hat{h}(p)$ such that there exists a closed subinterval $I'=[a,b]\subset I$, for $0<a<b$, such that $V_n\subset I'\times \R$ for $n$ sufficiently big. 
By Theorem \ref{pavel}, $h$ has non-tangential limit (finite or infinite) at $p$. Also $\lim_{M\ni z\to p}h(z)$ exists (finite or infinite). Indeed, $\Re h(z)=0$  and $\Im h(\phi_t(z))=\Im h(z)+t$ for all $z\in \hat M$ and all $t\geq 0$, proving that $\Im h(z)$ decreases when $z$ moves on $\hat M$ toward the initial point $p$ of $\hat M$. By \cite[Lemma 2.10.(1)]{BrGu}, $\lim_{M\ni z\to p}h(z)=\angle\lim_{z\to p}h(z)$. In particular, $\angle\lim_{r\to 1^-}\Re h(rp)=0$.  Since for every $n\in \N$ the curve $(0,1)\ni r\mapsto h(rp)$ is eventually contained in $V_n$, and since $V_n\subset \{\Re z>a\}$,  it follows that $\lim_{r\to 1^-}\Re h(rp)\geq a>0$, a contradiction. Therefore (L) implies $p\not\in\overline{\hat M}$.

Now assume  $p\not\in\overline{M}$.  

Suppose first $p$ is a boundary regular fixed point of $(\phi_t)$. By Proposition \ref{compon-bound}, there exist $-\infty<a<b<+\infty$ such that $(a,b)\subset I$, $\{z\in \C: a<\Re z<b\}\subset h(\D)$ but $\{z:a-\epsilon<\Re z<b+\epsilon\}\not\subset h(\D)$ for any $\epsilon>0$, and for every $c\in (a,b)$, $\lim_{t\to -\infty}h(c+it)=p$. Let $C_n$ be the connected component of $\{\Im z=-n\}$ which intersects  $\{z: a<\Re z<b\}$. Then  $(C_n)$ is a null chain which represents $\hat{h}(p)$. If $I=\R$ then $(C_n)$ satisfies (L) and we are done. Assume $I=(\alpha, \beta)$, with $-\infty<\alpha\leq a<b\leq \beta\leq +\infty$. If $\alpha=a$, then clearly $p$ is the initial point of the exceptional maximal contact arc defined by $\{z\in \de \D: z=h^{-1}(a+it), t\in \R\}$, therefore $\alpha<a$. Similarly $b<\beta$. Hence $(C_n)$ satisfies condition (L).

Now assume $p$ is not a boundary regular fixed point and suppose by contradiction that condition (L) is not satisfied. Let $(C_n)$ be a null chain and the $V_n$'s as defined above. The impression of $\hat{h}(p)$ is given by $I(\hat{h}(p))=\cap_{n\in \N} \overline{V_n}$ (where the closure has to be understood in the Riemann sphere $\C\mathbb P^1$). It is known that $x\in I(\hat{h}(p))$ if and only if there exists a sequence $\{z_n\}\subset \D$ such that $z_n\to p$ and $h(z_n)\to x$. 

By Theorem \ref{pavel}.(2), there exists $a\in \R$ such that $\lim_{z\to p}\Re h(z)=a$. Therefore, if (L) does not hold, it follows that $a\in \de I$. In particular, if $I=\R$, we obtain a contradiction and hence (L) holds in this case. Thus we can assume without loss of generality that $I=(a,b)$, with $-\infty< a<b\leq +\infty$. 

Since $\limsup_{z\to p}\Im h(z)<+\infty$ by Theorem \ref{pavel}.(4), the impression of $\hat{h}(p)$ is given by
\[
I(\hat{h}(p))=\{z: \Re z=a, \Im z\in [k_1,k_2]\},
\]
for some $-\infty\leq k_1< k_2<+\infty$ or  $I(\hat{h}(p))$ is the point at infinity  in $\C\mathbb P^1$. Let $\Gamma=h([0,1)p)$. Such a curve has a limit $q\in \C\mathbb P^1$ by Theorem \ref{pavel}.(1). Moreover, since $\lim_{z\to p} \Re h(z)=a$, it follows that $\overline{\Gamma}$ intersects the line $\{\Re z=a\}$ at $q$. In case $q$ is the point at infinity, let us set $\Im q=-\infty$. With such a notation, since $\Gamma+it\subset h(\D)$ for all $t\geq 0$,  it follows that $\{z: \Re z=a, \Im z>\Im q\}\subset \de h(\D)$. Therefore $(\phi_t)$ has an exceptional maximal contact arc $\hat M$ given by $h^{-1}(\{(a+it) : t>\Im q\})$.
We claim that $p\in\overline{\hat M}$, reaching a contradiction. This follows at once from
\begin{equation}\label{limitexc}
\lim_{t\to (\Im q)^{+}}h^{-1}(a+it)=p.
\end{equation}
In order to prove \eqref{limitexc} , let $D(q,\frac{1}{n})$ be a disc in the spherical metric centered at $q$ and with radius $1/n$, $n\in \N$. Note that for all $n\in \N$, $\Gamma\cap D(q,\frac{1}{n})\neq \emptyset$. Therefore, for each $n\in \N$ we can find a Jordan curve $T_n:[0,1]\to \C$  such that $T_n([0,1])\subset D(q,\frac{1}{n})$, $T_n((0,1))\subset h(\D)$, $T_n(0)=a+it_n$ for some $t_n>\Im q$, $T_n(1)\in \Gamma$. Let $Q_n:=T_n(0,1)$.  Then, the spherical diameters of $Q_n$ tend to $0$. Moreover, $h^{-1}(Q_n)$ is a sequence of Jordan arcs in $\D$, whose Euclidean diameter tends to $0$ by the 
no Koebe arcs theorem (see, {\sl e.g.}, \cite[Corollary 9.1]{Pommerenke}). Hence \eqref{limitexc} holds. 
\end{proof}

\section{Extension of topological conjugation  for non-elliptic semigroups}\label{estensionenon}
\begin{lemma}\label{taucon}
Let $(\phi_t)$ and $(\v_t)$ be two non-elliptic semigroups in $\D$.  Let $(\Omega_1, h_1, z\mapsto z+ti)$ be the holomorphic model of $(\phi_t)$ and let $(\Omega_2, h_2, z\mapsto z+ti)$ be the holomorphic model of $(\v_t)$. Write $\Omega_j=I_j\times \R$, $j=1,2$, where $I_j$ is any of the intervals $(-\infty,0)$, $(0,+\infty)$, $(0,\rho)$, with $\rho>0$, or $\R$.  Then $(\phi_t)$ and $(\v_t)$ are topologically conjugated if and only if there exist an  homeomorphism $u: I_1\to I_2$ and a continuous function $v:I_1\to \R$ such that 
\begin{equation}\label{formatau}
\tau(z):=u(\Re z)+i(\Im z+v(\Re z))
\end{equation}
 satisfies $\tau(h_1(\D))=h_2(\D)$.
\end{lemma}
\begin{proof}
Let $Q_1:=h_1(\D)$ and let $Q_2:=h_2(\D)$. By Proposition \ref{conj-model}, $(\phi_t)$ and $(\v_t)$ are topologically conjugated if and only if there exists a homeomorphism $\tau: \Omega_1\to \Omega_2$  such that $\tau(z+it)=\tau(z)+it$ for all $t\in \R$ and $\tau(Q_1)=Q_2$. 

Write $z=x+iy$ with  $x\in I_1$ and $ y\in \R$. Then $\tau(x+iy)=\tau(x)+iy$. Write $\tau(x)=u(x)+i v(x)$. Then $u: I_1\to I_2$ and $v: I_1\to \R$ are continuous. Since also $\tau^{-1}: \Omega_2 \to \Omega_1$ satisfies $\tau^{-1}(z+it)=\tau^{-1}(z)+it$ for all $t\in \R$, it follows that $u$ is a homeomorphism.
\end{proof}

Let $(\v_t)$ be a non-elliptic semigroup, let us denote by $M(\v_t)\subset \de\D$  the set of points which belong to an exceptional maximal contact arc for $(\v_t)$  (note that $M(\v_t)$ is open and possibly empty). Let $x\in \de\D$ be the Denjoy\,--\,Wolff point of $(\v_t)$ and let
\[
E(\v_t):=\overline{M(\v_t)\cup \{x\}}.
\]

\begin{proposition}\label{estensionebuona}
Let $(\phi_t)$ and $(\v_t)$ be two  non-elliptic semigroups of holomorphic self-maps of $\D$. Suppose $(\phi_t)$ and $(\v_t)$ are topologically conjugated via the homeomorphism $f:\D\to \D$. Then $f$ extends to a homeomorphism
\[
f:\oD\setminus E(\phi_t)\to \oD\setminus E(\v_t).
\]
Moreover, for all $p\in \oD\setminus E(\phi_t)$ it holds $T(p)=T(f(p))$ and $f(\phi_t(p))=\v_t(f(p))$ for all $t\geq 0$.
\end{proposition}

\begin{proof}
In order to prove the result, we use Carath\'eodory's prime ends theory (see, {\sl e.g.} \cite[Chapter 9]{CL}, \cite[Sections 2.4-2.5]{Pommerenke2}). Let $(\Omega_1=I_1\times\R, h_1, z\mapsto z+it)$ be the holomorphic model of $(\phi_t)$ and let  $(\Omega_2=I_2\times\R, h_2, z\mapsto z+it)$ be the holomorphic model of $(\v_t)$, where $I_1, I_2$ are intervals of the form  $(-\infty,0)$, $(0,+\infty)$, $(0,\rho)$, with $\rho>0$, or $\R$. Let $\hat{h}_1$ denote the homeomorphism from $\de \D$ to the Carath\'eodory prime-ends boundary of $h_1(\D)$ defined by $h_1$. 

Let $p\in \de \D\setminus E(\phi_t)$. Let $(C_n)$ be a null-chain in $h_1(\D)$ which represents $\hat{h}_1(p)$ and denote by $V_n$ the connected component of $h_1(\D)\setminus C_n$ which does not contain $C_0$. Let $I(\hat{h}_1(p))=\cap_{n>0} \overline{V_n}$ denote the impression of $\hat{h}_1(p)$ in the Riemann sphere (here $\overline{V_n}$ denotes the closure of $V_n$ in the Riemann sphere). Then, for any sequence $\{z_n\}\subset \D$ converging to $p$, the sequence $\{h(z_n)\}$ is eventually contained in $V_m$ for all $m\geq 0$. Conversely, if $\{w_n\}\subset h_1(\D)$ is any sequence which is eventually contained in $V_m$ for any $m\in \N$, then $\lim_{n\to \infty}h_1^{-1}(w_n)=p$.

Since $\limsup_{z\to p}\Im h_1(z)<K$ for some constant $K<+\infty$ by Theorem \ref{pavel}.(4), it follows that 
\begin{equation}\label{limitoh}
\sup\{\Im z: z\in I(\hat{h}_1(p))\}<+\infty.
\end{equation}

By Proposition \ref{prime-exc}, we can assume that $V_n\subset I_1'\times\R$ for some compact subinterval $I'\subset I$ and for all $n\in \N$. By Lemma \ref{taucon}, $f=h_2^{-1}\circ \tau \circ h_1$, where $\tau$ is given by \eqref{formatau}. In particular, setting $\tau(\infty):=\infty$, it follows that $\tau$ is uniformly continuous on $\overline{I_1'\times \R}$ (where the closure has to be understood in $\C\mathbb P^1$ and $\infty$ denotes the point at infinity). From this it follows easily that $\tau(C_n)$ is a null chain in $h_2(\D)$. Moreover, if $W_n$ is the connected component of $h_2(\D)\setminus\tau(C_n)$ which does not contain $\tau(C_0)$, then $W_n=\tau(V_n)$ and $W_n$ is contained in $u(I_1')\times \R$ for all $n\in \N$, with $u(I_1')$ being a compact subinterval of $I_2$. Let $q\in \de \D$ be such that $(\tau(C_n))$ represents $\hat{h}_2(q)$. Then $\tau(I(\hat{h}_1(p)))=I(\hat{h}_2(q))$, and by \eqref{limitoh}, 
\[
\sup\{\Im z: z\in I(\hat{h}_2(q))\}<+\infty,
\]
which, in turns, implies that $\limsup_{z\to q}\Im h_2(z)<+\infty$. Since  the non-tangential limit of $\Im h_2$ at the Denjoy\,--\,Wolff point of $(\v_t)$ is $+\infty$, this implies that $q$ is not the Denjoy\,--\,Wolff point of $(\v_t)$. Hence, by Proposition \ref{prime-exc}, $q$ does not belong to the closure of an exceptional maximal compact arc for $(\v_t)$. 

Now we show that, in fact, $\lim_{z\to p}f(z)=q$. Let $\{z_n\}$ be a sequence converging to $p$. Then $\{h_1(z_n)\}$ is eventually contained in $V_m$ for all $m\in \N$. Thus $\{\tau(h_1(z_n))\}$ is eventually contained in $\tau(V_m)=W_m$ for all $m\in \N$, which implies that $\{h_2^{-1}(\tau(h_1(z_n)))\} $ converges to $q$, that is, $\lim_{z\to p}f(z)=q$.

Next, if $p_1,p_2\in \de \D\setminus E(\phi_t)$ are two different points, the null chains $(C_n^1), (C_n^2)$ in $h_1(\D)$ which represents $\hat{h}_1(p_1)$ and $\hat{h}_1(p_2)$ are not equivalent (in the prime-ends sense). Therefore,  chosen those chains to satisfy condition (L) in Lemma \ref{prime-exc} for all $n\in \N$, it follows easily that $(\tau(C_n^1))$ and $(\tau(C_n^2))$ are not equivalent. Thus, if  $q_j\in \de \D$ is such that $(\tau(C_n^j))$ represents $\hat{h}_2(q_j)$, $j=1,2$, it follows that $q_1\neq q_2$. Hence, the extension of $f$ to $\de \D \setminus E(\phi_t)$ is injective.

Finally, we prove that $f:\oD\setminus E(\phi_t)\to \oD\setminus E(\v_t)$ is continuous. 

Recall the following fact from Carath\'eodory's prime ends theory. Let $h:\D \to h(\D)$ be univalent. Let $\{p_m\}$ be a sequence of points of $\de \D$. For each $m\in \N$, let $(C_n^m)$ be a null chain in $h(\D)$ representing the prime end $\hat{h}(p_n)$. Then $p_m\to p\in \de \D$ if and only if, given $(C_n)$ a null chain in $h(\D)$ representing $\hat{h}(p)$, for every open subset $U\subset h(\D)$ such that $C_n\subset U$ for all $n>>1$, it follows that there exists $m_0\in \N$ such that for every $m\geq m_0$ there exists  $N_m\in \N$ such that $C_n^m\subset U$ for all $m\geq m_0$ and $n\geq N_m$.

In order to prove continuity of $f$,  it is enough to show that if $\{p_m\}\subset \de\D\setminus E(\phi_t)$ converges to $p\in \de\D\setminus E(\phi_t)$ then $\lim_{m\to \infty}f(p_m)= f(p)$. 
Let $q_m:=f(p_m), q:=f(p)$. 

For each $m\in \N$, let $(C_n^m)$ be a null chain in $h_1(\D)$ representing $\hat{h}_1(p_m)$, chosen so that  condition (L) of Proposition \ref{prime-exc} is satisfied for all $n$, and similarly let $(C_n)$ be a null chain which satisfies condition (L) and represents $\hat{h}_1(p)$. 

As shown before, $(\tau(C_n^m))$ is a null chain in $h_2(\D)$ representing $\hat{h}_2(q_m)$ and $(\tau(C_n))$ is a null chain in $h_2(\D)$ representing $\hat{h}_2(q)$. 

Fix an open subset $U\subset h_2(\D)$ such that $\tau(C_n)\subset U$ for $n>>1$. Then, $\tau^{-1}(U)$ eventually contains $(C_n)$. Since $p_m\to p$, this implies that for every $m>>1$, $(C_n^m)$ is eventually contained in $\tau^{-1}(U)$. Hence $(\tau(C_n^m))$ is 
eventually contained in $U$ for $m>>1$. Therefore, $q_m\to q$ and $f$ is continuous.

Since the same applies to $f^{-1}$, it follows that $f:\oD\setminus E(\phi_t)\to \oD\setminus E(\v_t)$ is a homeomorphism.

In order to prove the last equations, fix $p\in \oD\setminus E(\phi_t)$ and fix $t\geq 0$. Let $r\in (0,1)$. Then $\v_t(f(rp))=f(\phi_t(rp))\to f(\phi_t(p))$ as $r\to 1$. Therefore, the limit of $\v_t$ along the continuous curve $r\mapsto f(rp)$ (which converges to $f(p)$) is $f(\phi_t(p))$. By the Lindel\"of theorem, $\v_t(f(p))=\angle\lim_{z\to f(p)}\v_t(z)=f(\phi_t(p))$, and thus the functional equation holds at $p$. From this,  it follows at once that $T(p)=T(f(p))$.
\end{proof}

\section{Topological invariants for non elliptic semigroups}\label{tinvariants}

First we examine maximal contact arcs which are not exceptional:

\begin{proposition}\label{corresp-arc}
Let $(\phi_t)$ and $(\v_t)$ be two non-elliptic semigroups in $\D$. Suppose $(\phi_t)$ and $(\v_t)$ are topologically conjugated via the homeomorphism $f:\D\to \D$. If $W$ is a maximal contact arc for $(\phi_t)$ which is not exceptional, then $f$ extends to a homeomorphism from $W$ onto $f(W)$ and $f(W)$ is a maximal contact arc for $(\v_t)$ which is not exceptional. 
\end{proposition}

\begin{proof}
Since $W$ is not exceptional, then $W\subset \de\D\setminus E(\phi_t)$. Thus, by Proposition \ref{estensionebuona}, $f$ extends as a homeomorphism on $W$ and $f(W)\subset \de\D\setminus E(\v_t)$. Therefore, we are left to show that $f(W)$ is a maximal contact arc. 

Let $p\in W$. By Proposition \ref{estensionebuona}, $T(p)=T(f(p))>0$, hence $\v_s(f(p))=f(\phi_s(p))$, $s\in (0,T(p))$, belong to a maximal contact arc $W'\subset f(W)$ for $(\v_t)$. If $q\in f(W)\setminus W'$, let $x_0\in W$ be such that $f^{-1}(q)=\phi_s(x_0)$,  for some $s>0$. But then, again by Proposition \ref{estensionebuona}, $f(\phi_s(x_0))$, $s\in (0,T(x_0))$ is a contact arc which contains $W'$, contradicting its maximality. Thus $f(W)$ is a maximal contact arc.
\end{proof}

Now we consider boundary fixed points and their type. Again, the result is an application of Proposition \ref{estensionebuona}.

\begin{proposition}\label{corresp-punto}
Let $(\phi_t)$ and $(\v_t)$ be two non-elliptic semigroups in $\D$. Suppose $(\phi_t)$ and $(\v_t)$ are topologically conjugated via the homeomorphism $f:\D\to \D$. Let $p\in \de \D$ be a boundary  fixed point for $(\phi_t)$. Suppose that $p\not\in E(\phi_t)$. Then the unrestricted limit
\[
f(p):=\lim_{z\to p}f(z)\in \de \D
\]
exists and $f(p)\in \de\D \setminus E(\v_t)$ is a boundary  fixed point for $(\v_t)$. Moreover, 
\begin{enumerate}
\item if $p$ is a  boundary regular fixed point for $(\phi_t)$  then  $f(p)$ is a  boundary regular fixed point for $(\v_t)$, 
\item if  $p$ is a boundary super-repelling fixed point of first type (respectively  of second type) for $(\phi_t)$ then $f(p)$ is a a boundary super-repelling fixed point of first type (respectively  of second type) for $(\v_t)$,
\end{enumerate} 
\end{proposition}

\begin{proof}
First assume $p$ is a boundary fixed point which does not start an exceptional maximal contact arc. By Proposition \ref{estensionebuona}, $f$ extends continuously at $p$, $f(p)$ is not in the closure of an exceptional maximal contact arc and $T(p)=T(f(p))$. Therefore, $f(p)$ is a boundary fixed point for $(\v_t)$. 

In order to show that the type of $p$ is preserved  let $(\Omega_1=I_1\times \R, h_1, z\mapsto z+ti)$ be the holomorphic model of $(\phi_t)$ and let $(\Omega_2=I_2\times \R, h_2, z\mapsto z+ti)$ be the holomorphic model of $(\v_t)$. Let $Q_1:=h_1(\D)$ and let $Q_2:=h_2(\D)$. By Lemma \ref{taucon}, $(\phi_t)$ and $(\v_t)$ are topologically conjugated if and only if there exists a homeomorphism $\tau: \Omega_1\to \Omega_2$  given by \eqref{formatau}  and $\tau(Q_1)=Q_2$.  By Proposition \ref{conj-model}, $f=h_2^{-1}\circ   \tau\circ h_1$.
 
Then the restriction of $\tau$ to $A(\phi_t)$ induces a homeomorphism from $A:=A(\phi_t)$ to $\tilde{A}:=A(\v_t)$. Indeed, 
\begin{equation*}
\begin{split}
\tilde{A}&=\cap_{t\geq 0} (Q_2+it)=\cap_{t\geq 0} (\tau(Q_1)+it)\\&=\cap_{t\geq 0} \tau(Q_1+it)=\tau\left(\cap_{t\geq 0} (Q_1+it)\right)=\tau(A).
\end{split}
\end{equation*} 
In particular, every connected component of $\tilde{A}$ corresponds to exactly one connected component of $A$.

Since $\tau$ defines a homeomorphism between $A$ and $\tilde{A}$ and maps vertical lines $\{\zeta\in \C: \Re \zeta=c\}$ to vertical lines, it follows that to each connected component of $A$ with non-empty interior (respectively with empty interior) there corresponds exactly one  connected component of $\tilde{A}$ with non-empty interior (resp. with empty interior). 

From this, it follows immediately that if $p$ is a boundary regular fixed point (respectively a super-repelling fixed point of first type), then $f(p)$ is  a boundary regular fixed point (resp. a super-repelling fixed point of first type).
Moreover, if $p$ is a super-repelling fixed point of second type, then so is $f(p)$, for otherwise one can apply the previous conclusion to $f^{-1}$ at $f(p)$ to get a contradiction. 
Hence, (1) and (2) are proved.
\end{proof}

Another  trivial consequence of Proposition \ref{estensionebuona} is that  the continuity of a semigroup at a boundary point is a topological invariant:

\begin{proposition}
Let $(\phi_t)$ and $(\v_t)$ be two non-elliptic semigroups in $\D$. Suppose $(\phi_t)$ and $(\v_t)$ are topologically conjugated via the homeomorphism $f:\D\to \D$. Let $p\in \de \D\setminus E(\phi_t)$ be such that the unrestricted limit $\lim_{z\to p}\phi_t(z)$ exists for all $t\geq 0$. Then the unrestricted limit $f(p)=\lim_{z\to p}f(z)\in \de \D\setminus E(\v_t)$ and the unrestricted limit $\lim_{z\to f(p)}\v_t(z)$ exist for all $t\geq 0$.
\end{proposition}

\begin{proof}
By Proposition \ref{estensionebuona}, $f$ has unrestricted limit at $p$, $f(p)\not\in E(\v_t)$ and $f^{-1}$ extends continuously at $f(p)$, $f^{-1}(f(p))=p$. Since $\v_t(z)=f(\phi_t(f^{-1}(z)))$, $z\in \D$, the result follows.
\end{proof}

\section{Exceptional maximal contact arcs and Denjoy\,--\,Wolff point}\label{exceptionalb}

In this section we examine the behavior of the topological conjugation on exceptional maximal contact arcs and boundary Denjoy\,--\,Wolff points. 
The behavior of the topological intertwining map on an exceptional maximal contact arc can be quite wild.

We start with the following result:

\begin{proposition}\label{inizio-rego}
Let $(\phi_t)$ and $(\v_t)$ be two non-elliptic semigroups in $\D$. Suppose $(\phi_t)$ and $(\v_t)$ are topologically conjugated via the homeomorphism $f:\D\to \D$. 
\begin{enumerate}
\item If $p\in \de \D$ is a boundary regular fixed point for $(\phi_t)$ which starts an exceptional maximal contact arc, then the non-tangential limit $f(p)=\angle\lim_{z\to p}f(z)$ exists and $f(p)$ is a boundary regular fixed point for $(\v_t)$ (possibly the Denjoy\,--\,Wolff point of $(\v_t)$). In fact, $f(p)\in E(\varphi_t)$.
\item If $p\in \de \D$ is the Denjoy\,--\,Wolff point of $(\phi_t)$ then the non-tangential limit $f(p)=\angle\lim_{z\to p}f(z)$ exists and $f(p)\in \de \D$ is the Denjoy\,--\,Wolff point of $(\v_t)$.
\end{enumerate}
\end{proposition}

\begin{proof}
Let $(\Omega_1=I_1\times\R, h_1, z\mapsto z+it)$ be the holomorphic model of $(\phi_t)$ and let  $(\Omega_2=I_2\times\R, h_2, z\mapsto z+it)$ be the holomorphic model of $(\v_t)$, where $I_1, I_2$ are (possibly unbounded) open intervals in $\R$.

Assume $M$ is the exceptional maximal contact arc for $(\phi_t)$ such that $p$ is its initial point. By Proposition \ref{arc-koenig}, $I\neq \R$ and we can assume without loss of generality that $I_1=(0,\rho)$, with $-\infty<0<\rho\leq +\infty$ and that $h(M)=\{it, t\in \R\}$. 

Let $C$ be the connected component  of $A(\phi_t)$ which corresponds to $p$ (see Proposition \ref{compon-bound}). Thus there exists $0<b\leq \rho$ such that $\overline C=\{\zeta\in \C: 0\leq\Re \zeta\leq b\}$.  Since we are assuming that $p$ is not the Denjoy\,--\,Wolff point of $(\phi_t)$, from Proposition \ref{compon-bound}, it follows that $b<+\infty$. Note that $b=\rho$ if and only if $h(\D)=\Omega_1$, that is, $(\phi_t)$ is a (hyperbolic) group of automorphisms. 

By Lemma \ref{taucon}, $f=h_2^{-1}\circ \tau \circ h_1$, with $\tau(z)=u(\Re z)+i(\Im z+v(\Re z))$, where $u:I_1\to I_2$ is a homeomorphism. With no loss of generality we can assume that $u$ is orientation preserving, and we let $\tilde a=\lim_{x\to 0^+}u(x)$ and $\tilde b=\lim_{x\to b^-}u(x)$.

As remarked in the proof of Proposition \ref{corresp-punto}, $\tilde{C}:=\tau(C)$ is a connected component in $A(\v_t)$. Clearly, $\overline{\tilde{C}}=\{\zeta\in \C: \tilde{a} \leq \Re \zeta\leq\tilde{b}\}$. Let $p(\tilde{C})$ denote the associated boundary regular fixed point for $(\v_t)$. 

Let  $\gamma:[0,1)\to  \D$ be a continuous curve which converges  non-tangentially to $p$. 
By  \cite[Lemma 4.4]{CD},  there exist $a<a'<b'<b$ and $t_0<1$ such that $a'\leq \Re h_1( \gamma(t))\leq b'$ for all $t\in (t_0,1)$. Therefore, there exist $\tilde{a} <\tilde{a}'<\tilde{b}'<\tilde{b}$ such that $\tilde{a}'\leq \Re \tau(\Re(h_1(\gamma(t))))\leq \tilde{b}'$ for all $t\in(t_0,1)$. Therefore, the continuous curve $\tau\circ h_1\circ \gamma$ is all contained in the strip $\{\zeta\in \C: \tilde{a}'\leq \Re\zeta\leq \tilde{b}'\}$ and its imaginary part converges to $-\infty$. It follows then from \cite[Prop. 4.5]{CD} that 
\[
\lim_{t\to 1} f(\gamma(t))=\lim_{t\to 1}h_2^{-1}(\tau(h_1(\gamma(t)))=p(\tilde{C}).
\]
So,  there exists $\angle \lim_{z\to p}f(z)$ and its value is $p(\tilde{C})$, which is a boundary regular fixed point for $(\phi_t)$.

(2) In case $p$ is the Denjoy\,--\,Wolff point of $(\phi_t)$, the argument is similar and we omit the proof. 
\end{proof} 

\begin{remark}\label{iperb-fijo}
From the previous proof it follows that if $I_2$ is relatively compact in $\R$, then $\overline{\tilde{C}}=\{\zeta\in \C: \tilde{a} \leq \Re \zeta\leq\tilde{b}\}$, with $|\tilde{a}|, |\tilde{b}|<+\infty$. Therefore, by Proposition \ref{compon-bound}, in this case $f(p)$ is not the Denjoy\,--\,Wolff point of $(\v_t)$. That is,  $f(p)$ can be the Denjoy\,--\,Wolff point of $(\v_t)$ only if the semigroup is parabolic.  
\end{remark}

In the following examples we show that the unrestricted limit of the homeomorphism $f$ might no exist at the Denjoy\,--\,Wolff point or at boundary regular fixed points which start an exceptional maximal contact arc. 

\begin{example}[{\sl Non existence of unrestricted limit at Denjoy\,--\,Wolff point}] Let $Q=\{ z\in \strip: \Im z \, \Re z>1\}$ and let $h:\D\to Q$ be a Riemann map. Consider the semigroup $(\phi_t)$ defined by $\phi_t(z):=h^{-1}(h(z)+it)$, $t\geq 0$. The point $p=\lim_{t\to+\infty} h^{-1}(\frac{1}{2}+it)$ is the Denjoy\,--\,Wolff point of $(\phi_t)$. Let $v:(0,1)\to \R$ be defined by $v(x)=-1/x$ for all $x\in (0,1)$, and define $\tau (z):=\Re z+i(\Im z+v(\Re z))$. Write
$  \tilde Q= \tau(Q)=\{ z\in \strip: \Im z >0\}$ and let $\tilde h:\D \to \tilde Q$ be a Riemann map. The semigroup $(\varphi_t)$ defined by $\varphi_t(z):=\tilde{h}^{-1}(\tilde{h}(z)+it)$, $t\geq 0$ is topologically conjugated to $(\phi_t)$ via  $f:=\tilde h^{-1}\circ \tau\circ h$. We claim that the unrestricted limit of $f$ at $p$ does not exist. The point $q:=\angle \lim_{z\to p}f(z)$ is the Denjoy\,--\,Wolff point of $(\varphi_t)$. Notice that $h$ and $\tilde h$ can be extended to homeomorphisms of $\overline \D$ onto $\overline Q $ and $\overline {\tilde Q}$, respectively.  We also denote by $h$ and $\tilde h$  those extensions.
Take $z_n=h^{-1}\left(\frac{1}{n}+(1+n)i\right)\in \D$. By the properties of $h$, we deduce that the sequence $(z_n)$ converges to $p$. Moreover,
$$
f(z_n)=\tilde h^{-1}(\tau(h(z_n)))=\tilde h^{-1} \left(\frac{1}{n}+i\right)\to \tilde h^{-1}(i)\neq q.
$$
\end{example}

\begin{example}[{\sl Non existence of unrestricted limit at boundary regular fixed points starting an exceptional maximal contact arc}]
Let $h:\D \to \strip$ be the Riemann map given by $h(z)=\frac12+\frac{i}{\pi}\log \frac{1+z}{1-z}$. Let $(\phi_t)$ be the (semi)group of hyperbolic automorphisms of $\D$ defined by $\phi_t(z):=h^{-1}(h(z)+it))$, $t\geq 0$. Then $1=\lim_{\Im z\to +\infty}h^{-1}(z)$ is the Denjoy\,--\,Wolff point of $(\phi_t)$, while $-1=\lim_{\Im z \to -\infty}h^{-1}(z)$ is a boundary regular fixed point of $(\phi_t)$. Let  $\tau:\strip \to \strip$ the homeomorphism defined by $\tau(z)=\Re z+i(\Im z +\frac{1}{\Re z})$. Let $f:=h^{-1}\circ \tau\circ h$ be a homeomorphism of $\D$. Then $f\circ \phi_t=\phi_t\circ f$. By Proposition \ref{inizio-rego}, $f$ has non-tangential limit at $-1$ and it is easy to see that $\angle\lim_{z\to -1}f(z)=-1$, for instance by taking the limit of $f$ along the (non-tangential) curve $(0,+\infty)\ni s\mapsto h^{-1}(1/2-s)$ converging to $-1$.

Consider now the curve $\gamma:(0,1)\to \strip$ defined by $\gamma(s):=(1-s)+i(s-1)^{-1}$. Then $h^{-1}(\gamma(s))$ is a curve in $\D$ which converges (tangentially) to $-1$ as $s\to 1$. However, 
\[
\lim_{s\to 1}f(h^{-1}(\gamma(s))=\lim_{s\to 1}h^{-1}(\tau(\gamma(s)))=\lim_{s\to 1}h^{-1}(1-s)=h^{-1}(0)\neq -1.
\]
Therefore, $f$ is not continuous at $-1$.
\end{example}

As we already proved in Proposition \ref{top1}, every non-elliptic semigroup of holomorphic self-maps can be topologically conjugated to a hyperbolic one (with $\strip$ as model domain). Reversing this implication, it follows that every non-elliptic semigroup can be topologically conjugated to one whose model domain is $\R\times \R$ (that is, a parabolic semigroup of zero hyperbolic step). In this case, there exist no exceptional maximal contact arcs, which show that exceptional maximal contact arcs are not topological invariants. 

However, if one stays in the class of hyperbolic semigroups, the question whether an exceptional maximal contact arc is a topological invariant is less trivial, as the following example shows:

\begin{example}
Let $\Omega=\{ z\in \strip: \Im z \, \Re z>-1\}$ and let $h:\D\to \Omega$ be a Riemann map. For $t\geq 0$, let $\varphi_t(z):=h^{-1}(h(z)+it)$. It is easy to see that $(\v_t)$ is a semigroup of holomorphic self-maps of $\D$. The point $p=\lim_{t\to-\infty} h^{-1}(t-i/2t)$ is a boundary super repelling fixed point of third type for this semigroup. Let $v:(0,1)\to \R$ be a continuous function and define $\tau (z):=\Re z+i(\Im z+v(\Re z))$. Write
$  \Omega_v:= \tau(\Omega)=\{ z\in \strip: \Im z >-\frac{1}{\Re z}+v(\Re z)\}$. Let $h_v:\D \to \Omega_v$ be a Riemann map. Let $\phi_t(z):=h_v^{-1}(h_v(z)+it)$, $t\geq 0$. The semigroup $(\phi_t)$ is clearly topologically conjugated to $(\varphi_t)$. On the one hand, taking $v(x):=2/x$, we obtain a semigroup whose unique fixed point is its Denjoy\,--\,Wolff point and the map $f=h_v^{-1}\circ\tau\circ h $ sends the exceptional maximal contact arc starting  at $p$ to the Denjoy\,--\,Wolff point of $(\phi_t)$. In particular, the boundary super repelling fixed point of third type $p$ is sent to the Denjoy\,--\,Wolff point. On the other hand, taking $v(x)=1/x$ we obtain a semigroup whose unique fixed point is its Denjoy\,--\,Wolff point and the map $f=h_v^{-1}\circ\tau\circ h $ sends the exceptional maximal contact arc starting  at $p$  onto an exceptional maximal contact arc for $(\phi_t)$ having a  non-fixed point as initial point. In particular, the boundary super repelling fixed point of third type $p$ is sent to a contact, not fixed, point.
\end{example}

\begin{proposition}\label{top-exc-arc}
Let $(\phi_t)$ be  a hyperbolic semigroup of holomorphic self-maps of $\D$ with Denjoy\,--\,Wolff point $\sigma\in \de \D$. Let $M$ be an exceptional maximal contact arc for $(\phi_t)$ whose initial point is $x_0\in \de \D$.
\begin{enumerate}
\item if $x_0$ is a  regular boundary fixed point for $(\phi_t)$ and $(\v_t)$ is a hyperbolic semigroup of holomorphic self-maps of $\D$, topologically conjugated to $(\phi_t)$ via the homeomorphism  $f: \D \to \D$ then $(\v_t)$ has an exceptional maximal contact arc  whose initial point is $\angle\lim_{z\to x_0}f(z)$. 
\item If $x_0$ is not a boundary regular fixed point for $(\phi_t)$ then there exists a homeomorphism $f: \D \to \D$ such that $(\v_t:=f^{-1}\circ \phi_t\circ f)$ is a hyperbolic semigroup of holomorphic self-maps of $\D$ and $\lim_{z\to p}f(z)=\sigma$ for all $p\in M$.
\end{enumerate}
\end{proposition} 
\begin{proof} (1) By Remark \ref{iperb-fijo}, $f(x_0):=\angle\lim_{z\to x_0}f(z)$ is a boundary regular fixed point of $(\v_t)$ different from the Denjoy\,--\,Wolff point. If $f(x_0)$ does not belong to the closure of an exceptional maximal contact arc for $(\v_t)$, by Proposition \ref{estensionebuona},  $f^{-1}$ has unrestricted limit at $f(x_0)$ and 
\[
x_0=\lim_{(0,1)\ni r\to 1}f^{-1}(f(rx_0))=f^{-1}(f(x_0))\not\in E(\phi_t),
\]
a contradiction. Therefore $f(x_0)$ belongs to the closure of an exceptional maximal contact arc for $(\v_t)$ and it is in fact its initial point, being a fixed point.

(2) Let $(\Omega=(0,b)\times \R, h, z\mapsto z+ti)$, $0<b<+\infty$ be the holomorphic model of
$(\phi_t)$ and let $Q:=h(\D)$. Let $M$ be an exceptional maximal contact arc with initial point $x_0$. By Theorem \ref{chac_contact_arc}, we can assume without loss of generality that $\Re h(z)=0$ for all $z\in M$. 

Assume $x_0$ is not a boundary {\sl regular} fixed point for $(\phi_t)$.

If $x_0$ is a super-repelling boundary fixed point for $(\phi_t)$, let
\[
\lambda_n=\sup\{y: s+iy\not\in Q, s\in (0,1/n]\}.
\]
 Since $x_0$ is
 super-repelling, by Theorem \ref{pavel}, $\lim_{r\to 1} \Re h(rx_0)=0$ and
$\lim_{r\to 1} \Im h(rx_0)=-\infty$. Thus $\lambda_n$ tends to $-\infty$. For
each $n\in \N$, take $0<s_n\leq 1/n$ and $y_n$ such that $s_n+iy_n\not\in Q$ and
$y_n\geq \lambda_n-1/n$. Take a monotone subsequence $(s_{n_k})$ of
$(s_n)$ and a continuous function $v:I\to I$ such that
$v(s_{n_k})=-2\lambda_{n_k}$. Define $\tau: \Omega \to \Omega$ by $\tau(x+iy):=x+i(y+v(x))$. Note that $\tau$ is a homeomorphism. Let
$h_2:\D \to \tau(Q)$ be a Riemann map. Let  $(\v_t)$ be the semigroup in $\D$ defined by
$\varphi_t(z):=h_2^{-1}(h_2(z)+it)$. Clearly, $(\v_t)$  is topologically conjugated to $(\phi_t)$. Since
$w_n=\tau(s_n+iy_n)\not\in \tau(Q)$ and 
\[
\lim_k\Im (w_{n_k})=\lim_k(y_{n_k}+v(s_{n_k}))\geq \lim_k(-\lambda_{n_k}+2/{n_k})=+\infty,\] 
it follows that $f$ maps the exceptional maximal contact arc $M$ to the Denjoy\,--\,Wolff point of $(\v_t)$.

Next, suppose that $x_0$ is not a boundary fixed point. Then, by
Theorem \ref{pavel}, there exists $\lim_{r\to 1} h(rx_0)= iy_0$ with $y_0\in
\R$. Therefore there exist points $x_n+iy_n\in\Omega\setminus Q$ such $x_n$ goes to $0$
and $y_n$ goes to $y_0$. Take a continuous function $v:I\to I$ such that
$\lim_{x\to 0}v(x)=+\infty$. As before, define $\tau:\Omega\to\Omega$ by $\tau(x+iy):=x+i(y+v(x))$ and let
$h_2:\D\to\tau(Q)$ be any Riemann map. By construction, $\tau
(x_n+iy_n)\in \Omega\setminus h_2(\D) $ and $\Im \tau (x_n+iy_n)$ goes to
$+\infty$. Therefore the  semigroup $(\v_t)$ defined by $\varphi_t(z):=h_2^{-1}(h_2(z)+it)$ for $t\geq 0$ is topologically conjugated to $(\phi_t)$ and  $f$ maps the exceptional maximal contact arc $W$ to the Denjoy\,--\,Wolff point of $(\v_t)$.
\end{proof} 

The $\omega$-limit of $f$ at  an exceptional maximal contact arc is described by the following proposition:

 \begin{proposition} Let $(\phi_t)$ and $(\v_t)$ be two non-elliptic semigroups in $\D$ with holomorphic models $(\Omega_1=I_1+i\R, h_1, z\mapsto z+ti)$ and  $(\Omega_2=I_2+i\R, h_2, z\mapsto z+ti)$, respectively. Suppose $(\phi_t)$ and $(\v_t)$ are topologically conjugated via the homeomorphism $f:\D\to \D$. Let $p$ and $q$ be the Denjoy\,--\,Wolff points  of $(\phi_t)$ and $(\varphi_t)$, respectively. Assume that $W$ is an exceptional maximal contact arc for $(\phi_t)$. Denote by $S=\{ z\in \D\setminus\{0\}: z/|z|\in W\}$ and let $x_0$ be the initial point of $W$. Then the set
$$
E(W)=\left\{ w\in \overline \D: \exists \{z_n\}  \subset S,  z_n \to z\in \overline W,   f(z_n)\to w\right\}
$$
is a compact connected arc in $\partial \D$ containing  $q$ which is contained in an exceptional maximal contact arc for $(\v_t)$.
   \end{proposition}
  \begin{proof} 
   Let $Q_1:=h_1(\D)$ and let $Q_2:=h_2(\D)$. By Lemma \ref{taucon}, $(\phi_t)$ and $(\v_t)$ are topologically conjugated if and only if there exists a homeomorphism $\tau: \Omega_1\to \Omega_2$  given by \eqref{formatau}  and $\tau(Q_1)=Q_2$.  By Proposition \ref{conj-model}, $f=h_2^{-1}\circ   \tau\circ h_1$.
  
  Since $\angle \lim_{z\to p} f(z)=q$ by Proposition \ref{inizio-rego}, it follows that $q\in E(W)$. Moreover, it is easy to see that $E(W)$ is a compact subset of $ \partial \D$. Therefore  we are left to check that $E(W)$ is connected. Indeed, assume this is not the case. Then there exist two compact sets $A$ and $B$ such that $E(W)=A_1\cup A_2$ and $A_1\cap A_2=\emptyset$. Write $k=d(A,B)>0.$ For $j=1,2$, take $w_j\in A_j$, $z_{n,j}\in \D$, for all $n$,  with  $ \frac{z_{n,j}}{|z_{n,j}|}\in \overline W,\   z_{n,j} \to z_j\in \overline W$ and  $f(z_{n,j})\to w_j$. We may assume that $d(f(z_{n,j}),w_j)<k/4$. 
 Let $C_n$ be the arc in $\overline W$ that joins $\frac{z_{n,1}}{|z_{n,1}|}$ with $\frac{z_{n,2}}{|z_{n,2}|}$ and 
  $$
  \Gamma_n= [z_{n,1},r_n z_{n,1}]\cup r_n C_n\cup [z_{n,2},r_n z_{n,2}]
  $$
  where $r_n=\max\{|z_{n,1}|,|z_{n,2}|\}$. Notice that $\Gamma_n$ is  connected. Consider the continuous  function $l:\Gamma_n\to \R$ given by $l(z)=d(f(z),A_1)$. Since $l(z_{n,1})<k/4$ and $l(z_{n,2})\geq d(w_2,A_1)-d(z_{n,2},w_2)\geq k-k/4=3k/4$. So there is $\alpha_n\in  \Gamma_n$ such that $d(f(\alpha_n),A_1)=k/2$. Since $|\alpha_n|\geq \min\{|z_{n,1}|,|z_{n,2}|\}$, we can take a subsequence such that $\alpha_{n_k}\to z\in W$ and $f(\alpha_{n_k})\to w$. Clearly, $w\in E(W)$ and $d(f(w),A_1)=k/2$. A contradiction. Hence $E(W)$ is connected. Clearly,  $E(W)$ is contained in an exceptional maximal contact arc for $(\v_t)$, for otherwise $f^{?1}$ would map points of $\de\D\setminus E(\v_t)$ into $E(\phi_t)$. \end{proof}

  \begin{example} Let $\Omega_1=\{ z\in \strip: \Im z >0\}$ and $h_1:\D\to \Omega_1$  a Riemann map. Consider the semigroup $(\phi_t)$ defined by $\phi_t(z):=h_1^{-1}(h_1(z)+it)$, $t\geq 0$. The arc $W=h_1^{-1}([0,\infty)i)$ is an exceptional maximal contact arc for $(\phi_t)$. Define  $\tau:\strip\to\strip$ as $\tau(x+iy):=x+i(y-1/x)$, 
  $
  \Omega_2:=\tau(\Omega_1)=\{ z\in \strip: \Im z \, \Re z>-1\}
  $
 and let $h_2:\D\to \Omega_2$ be  a Riemann map. The semigroup  $(\v_t)$ defined by $\varphi_t(z):=h_2^{-1}(h_2(z)+it)$, $t\geq 0$, is topological conjugated to $(\phi_t)$ via the homeomorphism $f=h_2^{-1}\circ   \tau\circ h_1$. The semigroup $(\v_t)$ has  an exceptional maximal contact arc $\tilde W=h_2^{-1}(\R i)$ with initial point a fixed point $x_0$. Notice that for all $w\in W$ it holds  $x_0=\lim_{z\to w}f(z)$, namely, the map $f$ sends the arc $W$ to the point $x_0$. This does not contradict the previous proposition, since $\tilde W$ is  the $\omega$-limit of $f$ at the Denjoy\,--\,Wolff point of $(\phi_t)$ .
\end{example}

\section{Elliptic case}\label{passtoelliptic}

In this final section we show  how to recover the results of Sections \ref{estensionenon} and \ref{tinvariants} in case of elliptic semigroups which are not groups. The key is to replace Lemma \ref{taucon} by the following:

\begin{lemma}
Let $(\phi _{t})$ and $(\varphi _{t})$ be two elliptic semigroups
in $\mathbb{D}$, which are not  groups. Let $(\mathbb{C},h_{1},z\mapsto e^{\lambda _{1}t}z)$ be a
holomorphic model of $(\phi _{t})$ and let $(\mathbb{C},h_{2},z\mapsto
e^{\lambda _{2}t}z)$ be a holomorphic model of $(\varphi _{t}).$ Then, $%
(\phi _{t})$ and $(\varphi _{t})$ are topologically conjugated if and only
if there exist a homeomorphism $u$ of the unit circle $\partial \mathbb{D}$
and a continuous map $v:\partial \mathbb{D\rightarrow (}0,\mathbb{+\infty )}$
such that $\tau :=\theta _{\lambda _{2}}^{-1}\circ \tau _{0}\circ \theta
_{\lambda _{1}}$ satisfies $\tau (h_{1}(\mathbb{D}))=h_{2}(\mathbb{D})$,
where%
\[
\tau _{0}(z):=\left\{
\begin{array}{ll}
0, & z=0 \\
|z|u(z/|z|)v(z/|z|), & z\neq 0%
\end{array}%
\right. ,\text{ }z\in \mathbb{C},
\]
and, given $\lambda =a+ib$ with $a<0,$%
\[
\theta _{\lambda }(z):=z|z|^{-(1+1/a)}\exp \left( -i\frac{b}{a}\mathrm{Log}%
|z|\right) ,z\in \mathbb{C}^{\ast }\text{ and }\theta _{\lambda }(0):=0.
\]
\end{lemma}

As already remarked, if the semigroup $(\phi_t)$ is elliptic, the set $E(\phi_t)=\emptyset$. Using the previous lemma and mimicking the proof of  Proposition \ref{estensionebuona}, one can prove the following extension result:

\begin{proposition}\label{estensionebuona_elliptic}
Let $(\phi_t)$ and $(\v_t)$ be two elliptic semigroups of holomorphic self-maps of $\D$, which are not groups. Suppose $(\phi_t)$ and $(\v_t)$ are topologically conjugated via the homeomorphism $f:\D\to \D$. Then $f$ extends to a homeomorphism
\[
f:\oD\to \oD .
\]
Moreover, for all $p\in \partial \D$ it holds $T(p)=T(f(p))$ and $f(\phi_t(p))=\v_t(f(p))$ for all $t\geq 0$. 
\end{proposition}

And also, 

\begin{proposition}
Let $(\phi_t)$ and $(\v_t)$ be two elliptic semigroups in $\D$, which are not groups. Suppose $(\phi_t)$ and $(\v_t)$ are topologically conjugated via the homeomorphism $f:\D\to \D$. 
\begin{enumerate}
\item If $W$ is a maximal contact arc for $(\phi_t)$, then $f$ extends to a homeomorphism from $W$ onto $f(W)$ and $f(W)$ is a maximal contact arc for $(\v_t)$. 
\item  Let $p\in \de \D$ be a boundary  fixed point for $(\phi_t)$.  Then the unrestricted limit
\[
f(p):=\lim_{z\to p}f(z)\in \de \D
\]
exists and $f(p)\in \de\D$ is a boundary  fixed point for $(\v_t)$. Moreover, 
\begin{enumerate}
\item if $p$ is a  boundary regular fixed point for $(\phi_t)$  then  $f(p)$ is a  boundary regular fixed point for $(\v_t)$, 
\item if  $p$ is a boundary super-repelling fixed point of first type (respectively  of second type) for $(\phi_t)$ then $f(p)$ is a a boundary super-repelling fixed point of first type (respectively  of second type) for $(\v_t)$.
\end{enumerate} 
\end{enumerate}
\end{proposition}

We end this section showing that Proposition \ref{estensionebuona_elliptic} is no longer true for groups of elliptic automorphisms. 
\begin{example} Consider the group of automorphisms $(\phi_t)$ where $\phi_t(z)=e^{it}z$ for all $t\in \R$ and $z\in \D$ and the continuous function $f:\D\to \D$ given by $f(z)=z \exp(i \ln(1-|z|))$ for all $z\in \D$. It is clear that $f$ is an homeomorphism of the unit disc with inverse function  $f^{-1}(z)=z \exp(-i \ln(1-|z|))$ and 
$$
f(\phi_t(z))=\phi_t(f(z)) \quad t\geq 0, \ z\in \D.
$$
But $f$ has no continuous extension at any point of $\partial \D$.
\end{example}

\end{document}